\documentclass[a4paper]{BDVR}
\usepackage[totalheight=25 true cm, totalwidth=16.5 true cm]{geometry}

\usepackage{color}
\usepackage{amsthm,amsmath,amsfonts,latexsym,url}
\usepackage{bm,enumitem}
\usepackage[colorlinks=true,citecolor=blue,urlcolor=blue,linkcolor=true, hypertexnames=false]{hyperref}
\usepackage{xcolor}
\usepackage[utf8]{inputenc}

\usepackage{charter}
\usepackage{microtype}
\usepackage{mathtools}
\usepackage[capitalize]{cleveref}

\usepackage[textwidth=1.7cm]{todonotes}

\newtheorem{definition}{Definition}

\newtheorem{notation}[definition]{Notation}
\newtheorem{exa}[definition]{Example}

\newtheorem{lemma}[definition]{Lemma}
\newtheorem{proposition}[definition]{Proposition}
\newtheorem{theorem}[definition]{Theorem}
\newtheorem{corollary}[definition]{Corollary}
\newtheorem{conjecture}[definition]{Conjecture}

\newtheorem{remark}[definition]{Remark}

\numberwithin{equation}{section}

\newcommand\blfootnote[1]{%
  \begingroup
  \renewcommand\thefootnote{}\footnote{#1}%
  \addtocounter{footnote}{-1}%
  \endgroup
}

\makeatletter
\def\iddots{\mathinner{\mkern1mu\raise\p@
\vbox{\kern7\p@\hbox{.}}\mkern2mu
\raise4\p@\hbox{.}\mkern2mu\raise7\p@\hbox{.}\mkern1mu}}
\makeatother

\def\cF{\mathcal{F}}
\def\cG{\mathcal{G}}

\def\cR{\mathcal{R}}

\def\fB{\mathfrak{B}}

\def\A{\mathbb{A}}

\def\C{\mathbb{C}}

\def\F{\mathbb{F}}

\def\K{\mathbb{K}}
\def\L{\mathbb{L}}

\def\N{\mathbb{N}}

\def\Q{\mathbb{Q}}

\def\Z{\mathbb{Z}}

\usepackage{comment}
\usepackage{xcolor}
\usepackage{csquotes}

\def\l{\left}
\def\r{\right}
\def\p{\prime}
\def\wtilde{\widetilde}

\def\ol{\overline}

\def\a{\alpha}

\def\Ga{\Gamma}

\def\Dt{D}

\def\GL{{\rm GL}}

\def\Aut{{\rm Aut}}

\DeclareMathOperator{\trdeg}{tr.deg}

\usepackage{endnotes}

\begin{document}

\begin{frontmatter}

\title{Differential transcendence of Bell numbers and relatives:    \\
a Galois theoretic approach\\
}

\author{Alin~Bostan}
\address{Inria, Universit\'e Paris-Saclay, \\ 1 rue Honor\'e d'Estienne d'Orves, 91120 Palaiseau, France}
\ead{alin.bostan@inria.fr}
\ead[url]{http://specfun.inria.fr/bostan}

\author{Lucia~Di Vizio}
\address{CNRS, Universit\'e Paris-Saclay, UVSQ, Laboratoire de mathématiques de Versailles, \\ 45 avenue des États-Unis, 78000 Versailles, France}
\ead{lucia.di-vizio@uvsq.fr}
\ead[url]{https://divizio.joomla.com}

\author{Kilian~Raschel}
\address{CNRS, Universit\'e Angers, Laboratoire Angevin de Recherche en Math\'ematiques,\\ 2 Boulevard Lavoisier, 49000 Angers, France}
\ead{raschel@math.cnrs.fr}
\ead[url]{https://raschel.perso.math.cnrs.fr/}

\begin{abstract}
\blfootnote{This project has received funding from the European Research
Council (ERC) under the European Union's Horizon 2020 research and innovation
programme under the Grant Agreement \href{http://combinepic.math.cnrs.fr/}{No
759702}, and from the ANR project
\href{https://mathexp.eu/DeRerumNatura/}{DeRerumNatura,
ANR-19-CE40-0018}.}In 2003 Klazar proved that the ordinary generating function of the sequence of Bell numbers is differentially transcendental over the field $\C(\{t\})$ of meromorphic functions at $0$.
We show that Klazar's result is an instance of a general phenomenon that
can be proven in a compact way using difference Galois theory.
We present the main principles of this theory in order to prove a general result about 
differential transcendence over $\C(\{t\})$, that we apply to many other (infinite classes
of) examples of generating functions, including as very special cases the ones
considered by~Klazar. 
Most of our examples belong to Sheffer's class, well studied notably in umbral calculus. 
They all bring concrete evidence in
support to the Pak-Yeliussizov conjecture, according to which a sequence whose
both ordinary and exponential generating functions satisfy nonlinear
differential equations with polynomial coefficients necessarily satisfies a
\emph{linear} recurrence with polynomial coefficients.
\end{abstract}

\date{\today}

\end{frontmatter}



\section{Introduction}\label{sec:intro}

In 
this paper we deal with what we will call 
the \emph{strong differential transcendence} of some power series
in~$\C[[t]]$, 
i.e., with their differential transcendence
over the field $\C(\{t\})$
of germs of meromorphic functions at~$0$. More precisely, we prove that the
solutions $f\in\C[[t]]$ of some first-order linear 
functional equations 
must be either rational, i.e., in $\C(t)$, or differentially transcendental over
$\C(\{t\})$, i.e., for any non-negative integer~$n$ and any polynomial
$P\in\C(\{t\})[X_0,X_1,\dots,X_n]$ we must have $P(f,f^\p,\dots,f^{(n)})\neq
0$, where $f^{(k)}$ denotes the $k$-th derivative of~$f$ with respect to $t$. 
We conclude in this way that many well-known power series in $\C[[t]]$ 
with a combinatorial origin are strongly differentially transcendental. 

\subsection*{State of the art}
We have come to consider this problem, influenced by three earlier works, by
Klazar~\cite{Klazar03}, by Pak~\cite{Pak18} and by Adamczewski, Dreyfus and
Hardouin~\cite{Adamczewski-Dreyfus-Hardouin}.

\smallskip 
First of all, in \cite{Klazar03}, Klazar considers the ordinary generating
function (OGF) of the Bell numbers $\phi(t)\coloneqq 1+\sum_{n\geq 1}\phi_nt^n$, where $\phi_n$
is the number of partitions of a set of cardinality $n\geq 1$, and proves that
$\phi(t)$ is differentially transcendental over $\C(\{t\})$.
To do so, he uses a functional equation satisfied by $\phi(t)$, namely:
\begin{equation}\label{eq:Bell-intro}
\phi\l(\frac{t}{1+t}\r)=t\phi(t)+1.
\end{equation}
A classical and important property of the Bell numbers is that their
\emph{exponential generating function} (EGF) 
\[\hat{\phi}(t) \coloneqq   1+\sum_{n\geq 1}\frac{\phi_n}{n!}t^n\] 
satisfies
\[
\hat{\phi}(t) =\exp(\exp t-1).
\]
As a consequence, $\hat{\phi}(t)$ is D-algebraic, meaning that it satisfies an algebraic differential equation over~$\Q(t)$
(or equivalently over $\C(t)$).
However, $\hat{\phi}(t)$ is not D-finite, that is, it does not satisfy any 
\emph{linear} differential equation with coefficients in~$\Q(t)$.             
This can be seen either analytically, using the asymptotics of $\phi_n$~\cite[Eq.~(5.47)]{Odlyzko95}, or algebraically, using~\cite{Singer86}
and the fact that the power series $\exp(t)$ is not algebraic.

\smallskip Secondly, starting from the example of the Bell numbers, Pak and Yeliussizov formulated the following ambitious conjecture as an ``advanced generalization of Klazar's theorem'':

\begin{conjecture}[{\cite[Open Problem 2.4]{Pak18}}]\label{conj:PY}
If for a sequence of rational numbers $(a_n)_{n \geq 0}$ both ordinary and
exponential generating functions
$\sum_{n\geq 0} a_n t^n$ and
$\sum_{n\geq 0} a_n \frac{t^n}{n!}$
are D-algebraic, then both are D-finite
(equivalently, $(a_n)_{n \geq 0}$ satisfies a linear recurrence with
polynomial coefficients in~$\Q[n]$).
\end{conjecture}

Thirdly, a very recent work by Adamczewski, Dreyfus and Hardouin in difference Galois theory establishes the following general statement: 

\begin{theorem}[{\cite[Thm.~1.2]{Adamczewski-Dreyfus-Hardouin}}]
\label{thm:ADH}
Let $f\in\C((t))$ be a Laurent series satisfying a linear functional equation of the form
\begin{equation*}
   \alpha_0 y+\alpha_1\tau(y)+\dots+\alpha_n\tau^n(y)=0,
\end{equation*}
where $\alpha_i\in\C(t)$, not all zero, and $\tau$ is one of the following operators:
\begin{itemize}
  \item $\tau(f(t))=f\l(\frac{t}{1+t}\r)$;
  \item $\tau(f(t))=f\l(qt\r)$ for some $q\in\C^\ast$, not a root of unity;
  \item $\tau(f(t))=f\l(t^m\r)$ for some positive integer $m$.
\end{itemize}
Then  either $f\in\C(t^{1/r})$ for some positive integer $r$, or $f$ is D-transcendental over $\C(t)$.
Moreover, in the case of the first operator, $r$ is necessarily equal to $1$.
\end{theorem}

The juxtaposition of the three works above raises immediately three remarks.
The first one is that Klazar is concerned with the \emph{strong} differential
transcendence, i.e.\ over $\C(\{t\})$, while Pak and Yeliussizov's conjecture and
the main theorem of \cite{Adamczewski-Dreyfus-Hardouin} are concerned with the
differential transcendence over the field of rational functions. 
The second one is that Theorem~\ref{thm:ADH} would prove
Conjecture~\ref{conj:PY}
if we were able to rephrase the 
differential properties of the series $\sum_{n\geq 0} a_n \frac{t^n}{n!}$ in terms of a difference operator
acting on $\sum_{n\geq 0} a_n t^n$. 
The third one is
that we could not find in the literature any another functional equation of
the form \eqref{eq:Bell-intro}, satisfied by other generating functions of
combinatorial interest, with the exception of the generating function of the
Bernoulli numbers, considered in~\cite{Zagier98}, and of Fubini numbers, as in~\cite{Klazar-ICM06}. In this paper, we address
these three issues.

\subsection*{Combinatorial examples for the Pak-Yeliussizov conjecture}

We consider families of polynomials $(P_n(x))_{n \ge 0}$ in $\C[x]$,
with $\deg(P_n)=n$, and
whose exponential generating functions (EGFs) 
are defined in terms of
power series $u,v,f,g,h \in \C[[t]]$ by
\begin{equation}\label{eq:expGF}
	\sum_{n \ge 0} P_n(x) \frac{t^n}{n!} =u(x)\, v(t)\,
f\big(g(x)h(t)\big).
\end{equation}
A surprisingly large number of 
{classical families of polynomials fit}
into the framework of
Eq.~\eqref{eq:expGF}.
An important special case is provided by \emph{Sheffer sequences}, very
classical in umbral calculus~\cite[Ch.~2]{Roman1984}, whose EGFs have the form
\begin{equation}\label{eq:Sheffer}
\sum_{n \geq 0} {P_n(x)} \frac{t^n}{n!} =
v(t) \, {\rm e}^{x \, h(t)}.
\end{equation}

\begin{table*}[!!!t]
\[
\begin{array}{c|l|c|c}
&\text{polynomial $P_n(x)$}
& v(t) &  h(t) \\
\hline
\hline
&\textcolor{blue}{\text{Laguerre}\ L_n^\alpha(x)}
&
\textcolor{blue}{(1-t)^{-1-\alpha}} &
\textcolor{blue}{-t(1-t)^{-1}}\\

&\textcolor{blue}{\text{Hermite}\ H_n(x)}
&
\textcolor{blue}{{\exp(-t^2)}} &
\textcolor{blue}{2t}\\

\text{D-finite}
&\textcolor{blue}{\text{Mott}\ M_n(x)}
&
\textcolor{blue}{1} &
\textcolor{blue}{{({\sqrt{1-t^2}-1})/t}}  \\

\text{examples}
&\textcolor{blue}{\text{Bessel}\ p_n(x)}
&
\textcolor{blue}{1} &
\textcolor{blue}{1-\sqrt{1-2t}} \\

&\textcolor{blue}{\text{Falling factorial}\ (x)_n}
&
\textcolor{blue}{1} &
\textcolor{blue}{\log(1+t)} \\

\hline

&\textcolor{red}{\text{Euler}\  E_n^{(\alpha)}(x)}
&
\textcolor{red}{{2^\alpha ({e^t+1})^{-\alpha}}} & \textcolor{red}{t} \\

&\textcolor{red}{\text{Bernoulli}\  B_n^{(\alpha)}(x)}
&
\textcolor{red}{t^\alpha(e^t-1)^{-\alpha}} & \textcolor{red}{t} \\

\text{exponential}
&\textcolor{red}{\text{Bell-Touchard}\ \phi_n(x)}
 & \textcolor{red}{1} & \textcolor{red}{\exp(t)-1} \\

\text{functions}
&\textcolor{red}{\text{Mahler}\ s_n(x)}
 & \textcolor{red}{1} & \textcolor{red}{1+t-\exp(t)} \\

&\textcolor{red}{\text{Actuarial}\ a_n^{(\beta)}(x)}
& \textcolor{red}{\exp(\beta t)} & \textcolor{red}{1-\exp(t)}\\

\hline
&\text{Bernoulli, 2nd kind}\ b_n(x)
&{{t}/{\log(1+t)}} & \log(1+t)\\

&\text{Poisson-Charlier}\ c_n(x;a)
& \exp(-t) & \log(1+t/a)\\

&\text{Narumi}\ N^{(a)}_n(x)
& {t}^a{\log(1+t)}^{-a} & \log(1+t)\\

\text{logarithmic}
&\text{Peters}\ P^{(\lambda,\mu)}_n(x)
&  {(1+(1+t)^\lambda)^{-\mu}}& \log(1+t)\\

\text{functions}
&\text{Meixner-Pollaczek}\  P_n^{(\lambda)}(x;\phi){~~}
&
(1+t^2-2 t \cos \phi)^{-\lambda}
& i\, \log \left( \frac{1-t e^{i \phi}}{1-t e^{-i \phi}} \right)\\

&\text{Meixner}\ m_n(x;\beta,c)
& {(1-t)^{-\beta}} &  \log \left( \frac{1-t/c}{1-t} \right)\\

&\text{Krawtchouk}\ K_n(x;p,N)
& (1+t)^N &  \log \left( \frac{p-(1-p)t}{p(1+t)} \right)\\
\end{array}
\]
\vskip-0ex\caption{Examples of various families of polynomials $P_n(x)$ of the Sheffer type~\cite[Ch.~2]{Roman1984}, with the corresponding $v,h$
as in~\eqref{eq:Sheffer}.
The first set of entries corresponds to D-finite examples.
We focus on the second set of entries, also called \emph{exponential functions}~\citep[\S19.7]{ErMaObIr1955}.
The last entries are sometimes called \emph{logarithmic functions}, and are not covered by our methods.}
\label{Fig:M1}
\end{table*}

\noindent An important subclass is that of \emph{Appell polynomials}~\cite{Appell1880}, for which $h(t)=t$.
Other interesting examples (see Table~\ref{Fig:M1}) include some families of
classical orthogonal polynomials (Hermite, Laguerre, Bessel, \ldots), for
which both ordinary and exponential generating functions are D-algebraic,
since they are even D-finite (always over~$\C(t)$ unless we clearly state otherwise).
Even more interesting examples, for our purpose, are those for which D-finiteness does not hold: for
instance the Bell-Touchard and the Bernoulli polynomials.
This is a consequence of the fact that their EGFs possess an infinite number of complex singularities, which is incompatible with D-finiteness.
In these cases, a
natural question is whether the corresponding OGF
\[F(x,t) = \sum_{n \geq 0} {P_n(x)} t^n\]
can still be D-algebraic, at least when evaluated at special values $x_0\in\C$ of~$x$.

As Klazar in his D-transcendence proof
for the OGF of the Bell numbers~\cite{Klazar03}
(which correspond to the evaluation at $x=1$ of the Bell-Touchard polynomials
$\phi_n(x)$ in Table~\ref{Fig:M1}), we focus on the case
where
$F(x,t)$ satisfies a functional equation
of the form
\begin{equation}\label{eq:tau}
F\left(x,\frac{t}{1+t}\right) = R(x,t) \cdot F(x,t) + S(x,t),
\end{equation}
where $R$ and $S$ are non-zero rational functions in $\C(x,t)$.
In \S\ref{subsec:framework}, we explain a \emph{recipe} to deduce a functional
equation \emph{\`a la Klazar} from the closed form of the EGFs in
Table~\ref{Fig:M2} (more precisely, from the differential equations
with exponential coefficients that they satisfy).
Therefore, all our examples bring further evidence and
reinforce Conjecture~\ref{conj:PY}, thanks to Theorem~\ref{thm:ADH}. 
This is the first contribution of our paper.
To our knowledge, Klazar's examples of the Bell and of the (related)
Uppuluri-Carpenter numbers were the only known combinatorial examples on which
Conjecture~\ref{conj:PY} was proved prior to our work.

\subsection*{Main Galois theoretic result}

Our aim in this article is to demonstrate, using difference Galois theory,
that Klazar's result is a very particular instance of a general phenomenon.
To do so, we equip~$\C$ with the usual absolute value so that it makes sense to consider the field $\C(\{t\})$,
which coincides with the field of fractions of the ring of convergent series at $0$ with coefficients in~$\C$.
We can finally state the second contribution of this paper (see \S\ref{sec:main} below),
which generalizes (from D-transcendence to strong D-transcendence) the first instance of Theorem~\ref{thm:ADH}, in the case of first-order inhomogeneous difference equations:

\begin{theorem}\label{thm:main-intro}
Let $\alpha, \beta, \gamma, \delta \in \C$ such that 
$\alpha\delta-\beta\gamma\neq 0$ and
$(\delta-\alpha)^2+4\gamma\beta=0$.
Let $f\in \C((t))\setminus \C(t)$ be any solution of
\[
f \left(\frac{\beta + \alpha t}{\delta + \gamma t}\right) = a(t)  f(t) + r(t),
\]
where $a, r \in \C(t)\setminus \{ 0 \}$.
Then $f$ is differentially transcendental over $\C(\{t\})$.
\end{theorem}

This result will be proven in \S\ref{sec:main} in the particular case 
$\alpha = \gamma = \delta=1$ and $\beta = 0$, which corresponds to the map $\tau$ in \cref{thm:ADH}.
Note that the hypotheses on $\alpha,\beta,\gamma,\delta$ are equivalent to 
assuming that the homography $t\mapsto\frac{\beta + \alpha t}{\delta + \gamma t}$ has only one fixed point, 
therefore Theorem~\ref{thm:main-intro} 
is an easy consequence of this particular case, by a rational change of variable, as any homography with only one fixed point is conjugated to $\tau$. See Corollary~\ref{cor:ADH++} below.

The theorem above applied to the OGFs of
Table~\ref{Fig:M2}, including all the families obtained after appropriately specializing
$x$ and~$\lambda $, allows to immediately obtain their strong differential transcendence:

\begin{corollary}\label{coro:main-intro}
Let $f(t)$ be any $F(x,t)$ in Table~\ref{Fig:M2}
evaluated at some $x\in \C\setminus \{ 0 \}$ and $\lambda \in \C$. 
Then $f(t)$ is differentially transcendental over $\mathbb{C}(\{t\})$.
\end{corollary}

Note that the assumption $x\in \C\setminus \{ 0 \}$ is only necessary for the last four entries 
of \cref{Fig:M2}.
Indeed, in the other cases the EGF is non D-finite for any $x \in \mathbb{C}$, hence the OGF as well, and therefore the OGF is strongly D-transcendental by \cref{thm:main-intro}.

\begin{table*}[!!!t]
\[
\renewcommand{\arraystretch}{1.5}
\begin{array}{l|c|c|c|c}
 \text{polynomial} \ P_n(x) &\text{EGF} \ \sum_{n\geq 0}P_n(x) \frac{t^n}{n!}
&  R & S &  \text{ref.} \\

\hline
\hline
{\text{Bernoulli}\  B_n(x)} &
\frac{t}{{\rm e}^{t}-1} \cdot \exp \left( x t\right)
& t+1 &  -{\frac { t \, \left( t+1 \right) }{ \left( t+1-xt \right) ^{2}}} 
& \text{\cite{Apostol08}}
\\     

{\text{Glaisher}\  U_n(x)} &
\frac{t}{{\rm e}^{t}+1} \cdot \exp \left( x t\right)
& -(t+1) &  {\frac { t \, \left( t+1 \right) }{ \left( t+1-xt  \right) ^{2}}} 
& \text{\cite[\S230,\S234]{Glaisher1898}}
\\

{\text{Apostol-Bernoulli}\  A_n^{(\lambda)}(x)} &
\frac{t}{\lambda \,{\rm e}^{t}-1} \cdot \exp \left( x t\right)
& \lambda \, (t+1) &  -{\frac { t \, \left( t+1 \right) }{ \left( t+1-xt  \right) ^{2}}} 
& \text{\cite{Apostol51}}
\\     

{\text{Imschenetzky}\  S_n(x)} &
\frac{t}{{\rm e}^{t}-1} \cdot \left(\exp(xt)-1\right)
& t+1 &  {\frac {x{t}^{2} \left( xt-2\,t-2 \right) }{ \left( t+1 \right)  \left( t+1-xt \right) ^{2}}} 
& \text{\cite{Imschenetzky1883},} 
\\
&&&& \text{\cite[\rm{p.~}254, (38)]{ErMaObIr1955}} 
\\

\hline

{\text{Euler}\  E_n(x)} & \frac{2}{{\rm e}^{t}+1} \cdot \exp \left( x t\right)
& -(t+1) & {\frac {2(t+1)}{t+1-xt}}
& \text{\cite{Carlitz58}} 
\\  

{\text{Genocchi}\  G_n(x)} & \frac{2 \, t}{{\rm e}^{t}+1} \cdot \exp \left( x t\right)
& -(t+1) & {\frac {2\, t\, \left( t+1 \right) }{ \left( t+1-xt \right) ^{2}}}
& \text{\cite{Horadam91}}
\\  
   
{\text{Carlitz}\ C_n^{(\lambda)}(x)}  & \frac{1-\lambda}{1-\lambda \, {\rm e}^{t}} \cdot \exp \left( x t\right)    
& \lambda \, (t+1) & {\frac {(1-\lambda)(t+1)}{t+1-xt}}
& \text{\cite{Carlitz53, Carlitz62}}
\\

{\text{Fubini}\ F_n(x)} & 1/(1-x({\rm e}^t-1))
& {\frac {x }{x+1}} \cdot \left( t+1 \right) & \frac{1}{x+1}
& \text{\cite{Tanny75}} 
\\

\hline
{\text{Bell-Touchard}\ \phi_n(x)} & \exp \left( x ( {\rm e}^{t}-1 )\right)
& xt & 1
& \text{\cite{Bell34,Touchard56}} 
\\

\hline

{\text{Mahler}\ s_n(x)} & \exp \left( x ( t + 1 - {\rm e}^{t} )\right)
& {\frac {x t \left( t+1 \right) }{xt-t-1}} & {\frac {t+1}{t+1-xt}}
& \text{\cite{Mahler32},}
\\
&&&& \text{\cite[\rm{p.~}254, (40)]{ErMaObIr1955}}
\\

{\text{Toscano's actuarial}\ a_n^{(\lambda )}(x)} &  \exp \left( {-x{{\rm e}^{t}}+\lambda  t+x}\right) 
&  {\frac {x t \left( t+1 \right) }{\lambda  t-t-1}}  & {\frac {t+1}{t+1-\lambda  t}}
& \text{\cite{Toscano50}, \cite[p.~42]{BoBu64}}
\\
\end{array}
\]
\vskip-0ex\caption{
Examples of power series $F(x,t) = \sum_{n\geq 0}P_n(x) {t^n}\in \C [x][[t]]$ with D-algebraic exponential generating functions (EGF) and satisfying first-order difference equations of the form 
$F \bigl(x, {\frac {t}{t+1}} \bigr) = R(x,t) \cdot F(x,t) + S(x,t)$, for some 
rational functions $R,S \in \C (x,t)$.
In the cases of Apostol-Bernoulli, Carlitz and Toscano, $\lambda $ is assumed to be a fixed parameter in $\C$.}
\label{Fig:M2}
\end{table*}                   
                   
\smallskip In particular, we deduce the strong D-transcendence of interesting combinatorial OGFs, among which the two main examples in Klazar's paper
(see \S\ref{subsec:combinat} for a few other combinatorial examples): 

\begin{itemize}[topsep=0pt,itemsep=-1ex,partopsep=1ex,parsep=1ex]
  \item the OGF of the \emph{Bell numbers}~\cite[Prop.~3.3]{Klazar03}  
(\href{https://oeis.org/A000110}{A000110})
\[
\sum_{n\geq 0} \phi_n(1) t^n = 
1+t+2\,{t}^{2}+5\,{t}^{3}+15\,{t}^{4}+52\,{t}^{5}+203\,{t}^{6}+\cdots
\]
   
 \item the OGF of the
\emph{Uppuluri-Carpenter numbers}~\cite[Thm.~3.5]{Klazar03}
(\href{https://oeis.org/A000587}{A000587})
\[
\sum_{n\geq 0} \phi_n(-1) t^n = 
1-t+{t}^{3}+{t}^{4}-2\,{t}^{5}-9\,{t}^{6}-9\,{t}^{7}+50\,{t}^{8}+\cdots
\]  

 \item the OGF of the
\emph{bicolored partitions}~\cite[Tab.~2]{BaBoDeFlGaGo02}     
(\href{https://oeis.org/A001861}{A001861}) 
\[
\sum_{n\geq 0} \phi_n(2) t^n = 
1+2\,t+6\,{t}^{2}+22\,{t}^{3}+94\,{t}^{4}+454\,{t}^{5}+
\cdots
\]

  \item the OGF  of the
number of \emph{set partitions without singletons} (\href{https://oeis.org/A000296}{A000296})
\[ 
\sum_{n\geq 0} s_n(-1) t^n = 
1+{t}^{2}+{t}^{3}+4\,{t}^{4}+11\,{t}^{5}+41\,{t}^{6}+162\,{t}^{7}+715\,{t}^{8}+\cdots
\]
         
 \item the OGF of the
\emph{Genocchi numbers}~\cite{Dumont74}
(\href{https://oeis.org/A001469}{A001469})
\[
\sum_{n\geq 0} G_n(1) t^n = 
t+{t}^{2}-{t}^{4}+3\,{t}^{6}-17\,{t}^{8}+155\,{t}^{10}-\cdots
\]  

\item \label{surj-nbs} the OGF of the {\emph{surjection numbers}} (also, \emph{preferential arrangements})~\cite[p.~109]{FlSe09} (\href{https://oeis.org/A000670}{A000670})  
\[
\sum_{n\geq 0} F_n(1) t^n = 
1+{t}+3\, {t}^{2}+13\,{t}^{3}+75\,{t}^{4}+541\,{t}^{5}+\cdots
\]                  
\end{itemize}

While Theorem~\ref{thm:main-intro} and Corollary~\ref{coro:main-intro} demonstrate that Klazar's theorem is an instance of a general phenomenon, they raise more questions than they answer. 
First, it is striking that so many concrete combinatorial objects are enumerated by strongly differentially transcendental functions.
Secondly, although there is a huge gap between the D-transcendental and the strongly differentially transcendental classes, Theorem~\ref{thm:main-intro}
and Corollary~\ref{coro:main-intro} show that their intersections with solutions of difference equations of order 1 coincide.
Thirdly, it is natural to inquire whether an extension of Conjecture~\ref{conj:PY} might hold with \emph{differentially algebraic
over~$\Q(t)$} replaced by \emph{differentially algebraic over~$\C(\{t\})$}.
We feel that this paper should provide a motivation to look
further into these questions.
\par
We should also point out that the (strong) D-transcendence of very natural combinatorial examples
such as 
the OGF of labeled rooted trees $\sum_{n\geq 1} n^{n-1} t^n$,
or the OGF of the logarithmic functions in Table~\ref{Fig:M1},
does not fit into our framework, and actually escapes for the moment any other attempt of proof. 
Three other challenging examples are presented in~\S\ref{ssec:challenges}.
\par
Finally, let us mention that there are several interesting examples satisfying higher order linear $\tau$-equations or linear $q$-difference equations, some of which arise from combinatorics of lattice walks~\cite{BeBoRa16,BeBoRa17,DHRS0,DHRS1,HaSo20}. We plan to consider them in subsequent publications under the viewpoint of the Pak-Yeliussizov conjecture, together with the differential transcendence with respect to the parameters $x$ and~$\lambda $ of the examples in Table~\ref{Fig:M2}.
We expect that similar techniques will allow us to obtain more general results, such as the algebraic-differential independence of
the families of power series in Table~\ref{Fig:M2}.

\subsection*{Content of the paper}

In Section~\ref{sec:setting}, we explain how to deduce a functional 
equation from a closed form of an exponential type such as the EGFs in Table~\ref{Fig:M2}. 
As a corollary, we find functional equations satisfied by 
several combinatorial examples, on which Theorem~\ref{thm:main-intro} 
will be applied. 
\par
Section~\ref{sec:DtransGalois}, and in particular \S\ref{subsec:Galois}, may be considered as a
quick and gentle
introduction to the parameterized Galois theory of difference equations,
starting from the point of view of the usual Galois theory of difference
equations.
In \S\ref{subsec:DiffTransc} we have included some proofs, since we have stated some results in the exact form
needed to prove Theorem~\ref{thm:main-intro}. 
They are quite similar to some statements in 
\cite{HardouinSinger}, but under a weaker assumption on the field of 
constants, necessary in our applications.
Besides, we have tried to provide a user-friendly exposition, 
avoiding as much as possible the use of the more sophisticated parameterized Galois theory, since we do not want to restrict the audience of the paper to specialists.
\par
Theorem~\ref{thm:main-intro} is our
main technical result. 
Its proof is given in Section~\ref{sec:main}, and it
relies on results from Section~\ref{sec:DtransGalois}.

\subsection*{Acknowledgments}
We thank Marni Mishna and Michael Singer 
for having suggested some additional bibliographic pointers.
We are grateful to Michael Singer for the enthusiasm that he showed for 
the questions tackled in this paper: this 
is encouraging for us to continue further in this direction. We also thank Christopher Chiu, Martin Klazar and Sergey Yurkevich for interesting comments on the first version of the manuscript.
We warmly thank the two anonymous referees for their detailed reports, their very interesting 
comments and their generous suggestions.


\section{From exponential generating functions to difference equations: setting and examples}
\label{sec:setting}

\subsection{From exponential generating functions to \texorpdfstring{$\tau$}{tau}-equations}
\label{subsec:framework}

Let $C$ be a field of characteristic zero. 
As before, we consider the substitution map $\tau:f(t)\mapsto f\l(\frac{t}{1+t}\r)$, with compositional inverse $\tau^{-1}:f(t)\mapsto f\l(\frac{t}{1-t}\r)$,
and 
the derivation  
$\Dt\coloneqq  \frac{d}{dt} : f(t) \mapsto f'(t)$. 
It is not difficult to see that $\tau$ 
defines an automorphism of $C((t))$, $C(\{t\})$ and $C(t)$.
We will informally call \emph{difference $\tau$-equation}, or simply \emph{$\tau$-equation}, a linear functional equation with respect to $\tau$.
\par
We want to tackle the relation between the closed forms of the EGFs of Table~\ref{Fig:M2} and the existence of a difference $\tau$-equation for the
corresponding OGFs.
First of all, we recall the definition of the formal Borel transform $\fB:C[[t]]\to C[[t]]$:
\[
g\coloneqq  \sum_{n\geq 0} g_n t^n \; \mapsto \; \hat g\coloneqq  \fB(g)=\sum_{n\geq 0} g_n \frac{t^n}{n!}.
\]
Moreover, we consider the Euler-type transform  $\Phi  : C[[t]]
\rightarrow C[[t]]$ defined by:
\[
f(t) \mapsto (\Phi  f)(t) \coloneqq  
\frac{1}{1-t}\cdot f\left( \frac{t}{1-t}\right).
\]
Note for later use that, for all $j \in \Z$, 
\begin{equation} \label{eq:iter}
	\Phi^j(f(t)) = \frac{1}{1-jt}\cdot f\left( \frac{t}{1-jt}\right)
	\qquad \text{and} \qquad
	\tau^j(f(t)) = f\left( \frac{t}{1+jt}\right).
\end{equation}

The maps $\fB$ and $\Phi $ intertwine in an interesting way, described in \cref{lemma:Borel} below,
reminiscent of the formal Fourier transform.
Although not stated in this generality, (variants of) this result can be found in various papers from various areas of mathematics and computer science, such as~\cite{Gould90,FlRi92,Schmid92,Prodinger94,FlSe95}. For instance, part \ref{it:(a)} is a consequence of Theorems~1 and 2 in \cite{Gould90}, which relate the coefficients of $f$ and those of $\Phi(f)$ via a binomial (Euler-type) transform. Section~1 in~\cite{FlSe95} gives a slight variation of~\ref{it:(a)}, while~\cite{Prodinger94} gives a generalization of~\ref{it:(a)} to transforms of the form $f(t) \mapsto \frac{1}{1-bt} \cdot f\left(\frac{ct}{1-bt} \right)$. The (inverse) Borel transform is used in~\cite{FlRi92,Schmid92} for the analysis of some ``digital trees'' and ``tree algorithms'', in order to solve different types of functional equations similar to ours (see Lemmas~1 and 2 in~\cite{FlRi92} and Eqs. (2.5) and (2.8) in~\cite{Schmid92}).

\begin{lemma}    \label{lemma:Borel}
For any $f,g\in C[[t]]$, we have:
\begin{enumerate}[label={\rm(\alph{*})},ref=(\alph{*})]
  \item\label{it:(a)}$\Phi (f)=g$ if and only if 
   $\fB(g) =  \fB(f) \cdot e^t$;
  \item\label{it:(b)}$\frac{d}{dt}(\fB(f))=\fB\l(\frac{f(t)-f(0)}{t}\r)$;
  \item\label{it:(c)}
  for any $j,k\in \N$, the following commutation rule holds
  \[
  e^{jt} \circ D^k \circ \fB = \fB \circ \Phi^j \circ \Delta^k,
  \]
where $\Delta : C[[t]] \rightarrow C[[t]]$ is the divided difference operator 
$ f(t) \mapsto (\Delta  f)(t) \coloneqq  
\frac{f(t) - f(0)}{t}.$
\end{enumerate}
\end{lemma}

\begin{proof}
Note that \ref{it:(a)} is equivalent to \ref{it:(c)} for the particular choice $(j,k) = (1,0)$ and that \ref{it:(b)} is equivalent to~\ref{it:(c)} for the particular choice $(j,k) = (0,1)$. Moreover, an easy induction on $j$ and $k$ shows that \ref{it:(c)} follows from an iteration of \ref{it:(a)} and \ref{it:(b)}.
It is therefore sufficient to prove~\ref{it:(a)} and \ref{it:(b)}. 
By linearity, it is enough to show that, for any $m\geq 0$, we have
  \begin{equation} \label{eq:tm}
  e^{t} \cdot (\fB (t^m)) = \fB(\Phi(t^m)) \quad \text{and} \quad
  D (\fB (t^m)) = \fB(\Delta(t^m)).
  \end{equation}
The second part of \cref{eq:tm} is straightforward, since both terms are clearly equal to $t^{m-1}/(m-1)!$. 
It remains to prove the first part of \cref{eq:tm}. Using the binomial theorem, its right-hand side is equal to
\[
\fB\left( \frac{1}{1-t} \cdot \left( \frac{t}{1-t}\right)^{m} \right)
=
\fB\left( t^{m} \cdot \sum_{n \geq 0} \binom{m+n}{n} t^n \right)
=
\sum_{n \geq 0} \frac{ t^{m+n}}{n! m!}
=
e^{t} \cdot \frac{t^{m}}{m!},
\]
and it is thus equal to its left-hand side.
This finishes the proof of the lemma.
\end{proof}

Note that \cref{lemma:Borel}\ref{it:(c)} also holds for non-integer values of $j$, by taking~\cref{eq:iter} as definition of $\Phi^j$. This simple remark will be used in some examples of~\cref{subsec:combinat}.

\cref{lemma:Borel} will be our main tool to transform any linear differential equation with polynomial coefficients in $\exp(t)$ satisfied by $\hat{f}=\fB(f)$ into a linear $\Phi$-difference equation satisfied by $f$.
Then, using the rule (consequence of~\eqref{eq:iter})
\begin{equation} \label{eq:comm_Phi_tau}
\tau^d \circ \Phi^j = \frac{1+dt}{1+(d-j)t} \circ \tau^{d-j}, \quad \text{for all} \; j, d \in \Z
\end{equation}
will allow us to transform the latter $\Phi$-difference equation into a $\tau$-difference equation satisfied by $f$.
Let us first illustrate this process on an example.

\begin{exa}
The case of Bell numbers, considered by Klazar and mentioned in \eqref{eq:Bell-intro}, coincides with the Bell-Touchard polynomials $\phi_n(x)$ in Table~\ref{Fig:M2} evaluated at $x=1$.
The associated EGF is $\hat{\phi}(t) = \exp({\rm e}^t -1)$, which satisfies
\[
\Dt(\hat{\phi})-e^t \cdot \hat{\phi}=0,\
\] 
thus $\fB \left( \frac{\phi-1}{t} \right) = \fB \left( \Phi (\phi) \right)$, by~\cref{lemma:Borel}\ref{it:(a)} and \cref{lemma:Borel}\ref{it:(b)}. 
Therefore, $\phi$ satisfies the first-order $\Phi$-equation
\[ \Phi (\phi) = \frac{\phi-1}{t}  . \]
Now, applying $\tau$ to both sides of this equation, and using \cref{eq:comm_Phi_tau} with $d=j=1$, we find
\[  (t+1) \phi(t) = \tau \left( \frac{\phi-1}{t} \right)  . \]
Since on the other hand
$\tau \left( \frac{\phi-1}{t} \right) =  \tau \left( \frac{1}{t} \right) \cdot \tau (\phi - 1) =
\frac{t+1}{t}  \cdot  \left( \tau (\phi) - 1 \right)$,
we finally conclude that $\phi$ satisfies the first-order $\tau$-equation
$\tau(\phi) = t\phi+1$, as expected. 
This provides an alternative proof to~\cite[Prop.~2.1]{Klazar03}.
\end{exa}

In general, an iteration of the same argument allows to show: 

\begin{proposition}\label{prop:diffeqtotau}
Let $f\in C[[t]]$. If $\hat{f} = \fB(f)$ satisfies a linear differential equation of order $r$, of the form
\begin{equation}\label{eq:deqexp}
a_0(e^t)\hat f+a_1(e^t)\Dt(\hat{f})+\dots+a_r(e^t)\Dt^r(\hat{f}) = 0,
\end{equation}
with $a_0,\dots,a_r\in C[t]$ of degree at most $d$, 
then $f$ satisfies a linear inhomogeneous difference $\tau$-equation of order
at most~$d$, with coefficients in~$C[t]$ of degree at most $d+r$.
\end{proposition}
\begin{proof}
Setting $a_i(t) = \sum_{j=0}^d a_{i,j} t^j$ with $a_{i,j} \in C$, \cref{eq:deqexp} writes 
$
\sum_{i=0}^r \sum_{j=0}^d a_{i,j} e^{jt} D^i (\fB(f)) = 0.
$
By \cref{lemma:Borel}, this implies 
$
\sum_{i=0}^r \sum_{j=0}^d a_{i,j} \Phi^j (\Delta^i(f)) = 0.
$
Applying $\tau^d$ to this last equation, and using~\eqref{eq:comm_Phi_tau}, yields
\begin{equation}\label{eq:tau-delta}
\sum_{j=0}^d  \sum_{i=0}^r  a_{i,d-j} \cdot \frac{\tau^j (\Delta^i(f))}{1+jt} = 0.
\end{equation}
Now, for $i>0$, each term $\tau^j (\Delta^i(f))$ in the above sum is equal to
\[
\Delta^i(f) \left(\frac{t}{1+jt} \right)
=
 \frac{f \left(\frac{t}{1+jt} \right) - R_{i,j} \left(\frac{t}{1+jt} \right)}{\left(\frac{t}{1+jt} \right)^i}
=
\frac{(1+jt)^{i}}{t^i} \cdot \left({ \tau^j(f) - R_{i,j} \left(\frac{t}{1+jt} \right)} \right),
\]
where $R_{i,j}\in C[t]$ is a polynomial of degree less than $i$.
Therefore, multiplying~\eqref{eq:tau-delta} by~$t^r \prod_{j=1}^d (1+jt)$ and arranging terms  yields an inhomogeneous $\tau$-equation of order at most $d$ with polynomial coefficients of degree at most $d+r$.
\end{proof}

\begin{remark}\label{remark:inhom_exppol}
\Cref{prop:diffeqtotau} generalizes without difficulty when \cref{eq:deqexp} has a non-zero right-hand side, e.g.\ an exponential polynomial $P \in C[t,e^t]$. In this case, the (equivalent) $\tau$-equation \eqref{eq:tau-delta} has a right-hand side equal to $\frac{\tau^d(\fB^{-1}(P))}{1+dt}$. For instance if $P\in C[t]$, then instead of multiplying~\eqref{eq:tau-delta} by just~$t^r \prod_{j=1}^d (1+jt)$, we multiply this time by $t^r (1+dt)^{d_P} \prod_{j=1}^d (1+jt)$, where $d_P \coloneqq \deg P$; this produces a $\tau$-equation of order at most $d$ and coefficients in $C[t]$ of degree at most $d_P + d + r$.
In the general case $P = \sum_{i=0}^{d_P} \sum_{j=0}^{e_P} p_{i,j} t^i e^{jt}$ with $p_{i,j} \in C$, we use that $\fB^{-1}(t^i e^{jt}) = i!t^i/(1-jt)^{i+1}$ to deduce, after multiplying this time the $\tau$-equation \eqref{eq:tau-delta} by $t^r (1+dt)^{\deg P}  \prod_{\ell=d-e_P}^{d-1} (1+\ell t)^{d_P+1} \prod_{j=1}^d (1+jt)$, a $\tau$-equation of order at most $d$ and with coefficients in $C[t]$ of degree at most $d_P + r+d + (d_P+1)e_P$.
\end{remark}

Before deducing from the proposition above the functional equations satisfied by the OGF in Table~\ref{Fig:M2}, 
let us just illustrate that one can obtain 
higher order difference equations in some other interesting cases:

\begin{exa}\label{exa:graphs}
We consider the EGF $\hat{f} = \exp(({\rm e}^t -1)^2/2)$. 
The associated OGF is $f = 
1+{t}^{2}+3\,{t}^{3}+10\,{t}^{4}+45\,{t}^{5}+\cdots$, 
the generating function of the numbers of simple labeled graphs on $n$ nodes in which each component is a complete bipartite graph, see \href{https://oeis.org/A060311}{A060311}.
Then
$\Dt(\hat{f})= \hat{f} \cdot {\rm e}^t ({\rm e}^t-1)$, 
thus $\fB(\frac{f-1}{t})= \fB(\Phi ^2(f) -  \Phi (f))$ by \cref{lemma:Borel}.
Therefore, $f$ satisfies the second-order $\Phi$-equation
$\Phi ^2(f) -  \Phi (f) = \frac{f-1}{t}$. Applying $\tau^2$ to both sides of this equation yields, using \eqref{eq:comm_Phi_tau}, the second-order $\tau$-equation
$(2t+1) f - \frac{2t+1}{t+1} \tau (f) = \frac{2t+1}{t} \cdot (\tau^2(f)-1)$.  
Rearranging terms, we deduce that $f$ satisfies the second-order $\tau$-equation
\end{exa}

\subsection{Other combinatorial examples} \label{subsec:combinat}

As promised in the introduction, we finish this section 
by considering the series in Table \ref{Fig:M2} plus a few more examples of $\tau$-equations arising from combinatorial generating functions. 

As far as Table \ref{Fig:M2} is concerned, we will only detail the case of the family of Bernoulli polynomials  
$B_n(x) = B_n^{(0)}(x)$ (see \cite{Apostol08}), defined by their exponential generating function as follows:
\[
\frac{t}{{\rm e}^{t}-1} \cdot \exp \left( x t\right) = \sum_{n \geq 0} B_n(x) \frac{t^n}{n!}.
\]
We apply our recipe to find a linear $\tau$-equation for
$B(x,t)$.

\begin{lemma}      \label{lem:Bernoulli}
The OGF of the Bernoulli polynomials $B(x,t)\coloneqq   \sum_{n \geq 0} B_n(x) t^n$ satisfies the functional equation 
\begin{equation}\label{eq:EqBernoulli}
\tau (B)=(t+1) \cdot B-\frac{t(t+1)}{(t+1-tx)^2}.
\end{equation}
\end{lemma}

\begin{proof} 
We start with the inhomogeneous differential equation (of order 0)
$ \hat{B} \cdot \exp(t) -  \hat{B} = t \cdot \exp(xt)$,
and follow the procedure described in the proof of \cref{prop:diffeqtotau}
and in \cref{remark:inhom_exppol}. 
We deduce $\fB(\phi(B) - B) = \fB(t/(xt-1)^2)$, 
and by applying $\tau \circ \fB^{-1}$ we get
$ (t+1) B - \tau(B) = t(t+1)/(t+1-xt)^2$.
\end{proof}

As a consequence of \cref{lem:Bernoulli}, \cref{thm:main-intro} implies \cref{coro:main-intro} for the first entry of \cref{Fig:M2}; in particular, this proves that the OGF $B(0,t)$ of the sequence of \emph{Bernoulli numbers} $B_n = B_n(0)$, classical in number theory, is strongly D-transcendental.
All the other functional equations for the OGFs in Table \ref{Fig:M2} are calculated in the same way,
by using the procedure described in the proof of \cref{prop:diffeqtotau}. In each case, \cref{thm:main-intro} implies \cref{coro:main-intro} for the corresponding entry of \cref{Fig:M2}.

\medskip We complete our list with a few more examples from combinatorics.

\subsubsection{Tangent numbers}\label{ssec:tangent}
The integer sequence of the so-called ``tangent numbers'' $(f_n)_{n \geq 0} = (1, 2, 16, 272, 7936, 353792,\ldots)$ (\href{https://oeis.org/A000182}{A000182}) appears in the expansion of the
(D-algebraic) tangent function as an EGF:
\[
\tan(t) =
 1 \, \frac{t}{1!}
+  2 \, \frac{t^3}{3!}
+  16 \, \frac{t^5}{5!} 
+  272 \, \frac{t^7}{7!}  
+  7936 \, \frac{t^9}{9!} 
+  353792 \, \frac{t^{11}}{11!} + 
\cdots .
\]

A slight variation of Lemma~\ref{lemma:Borel},
based on the exponential form of the tangent function, proves that the corresponding OGF  
$
F(t) = t+2\,{t}^{3}+16\,{t}^{5}+272\,{t}^{7}+7936\,{t}^{9}+   
\cdots
$
satisfies the difference equation
\begin{equation} \label{tau_eq_tan}
F \left( {\frac {t}{1-2\,it}} \right) + \left( 1-2\,it \right) F \left( t \right) =2\,t,
\end{equation}                                                       and Theorem~\ref{thm:main-intro} implies that $F(t)$
is strongly  D-transcendental.

Indeed, as pointed out earlier, \cref{lemma:Borel}\ref{it:(c)} also holds for non-integer values of $j$, by taking~\cref{eq:iter} as definition of~$\Phi^j$.
Thus, from the equality
\[
 \tan(t) = - i \, \frac{e^{2it}-1}{e^{2it}+1}
\]
we deduce that 
$
(e^{2it} + 1) \fB(F) = i (1 - e^{2it})
$, and hence
\[
\fB(\Phi^{2i}(F) + F) = \fB \left( i - \frac{i}{1-2it} \right),
\]
which in turn implies~\cref{tau_eq_tan}.

As a side remark, note that one can deduce from the D-transcendence of $F(t)$ an alternative proof for the  D-transcendence of the OGF of the Bernoulli numbers via the classical relation $f_n = 2^{2n+1} (2^{2n+2}-1) \frac{(-1)^n}{n+1} B_{2n+2}$.

\subsubsection{Alternating permutations}
A permutation $\sigma = \sigma_1 \cdots \sigma_n$ of the set $\{ 1,2,\ldots, n \} $ is called \emph{alternating} if $\sigma_1 < \sigma_2 > \sigma_3 < \cdots$.
Alternating permutations are counted by the sequence $(a_n)_{n\geq 0} = (1, 1, 1, 2, 5, 16, 61, 272,\ldots)$ (\href{https://oeis.org/A000111}{A000111}).   
By a famous result due to Andr\'e~\cite{Andre1881}, its EGF is known to be $\tan(t) + \sec(t)$, which is clearly D-algebraic.
{It is natural to inquire about the nature of the corresponding OGF, $A(t) \coloneqq   \sum_{n \geq 0} a_n t^n$.
Using the identity 
\[
\tan(t) + \sec(t) = \frac{1-i e^{it}}{e^{it}-i}
\]
we deduce as in \S\ref{ssec:tangent} that 
\[
\fB(\Phi^{i}(A) - A) = (e^{it} - i) \fB(A) = 1 - i \, e^{it} = \fB(1- i/(1-it)),
\]
and therefore 
$A(t)$  satisfies the difference equation
\[
A \left( {\frac {t}{1 - it}} \right) = \left( t + i \right) A \left( t \right) + 1-i-it.
\]                                                         
Theorem~\ref{thm:main-intro} then implies that $A(t)$ is strongly  D-transcendental.
               
\subsubsection{Springer numbers}
The sequence $(s_n)_{n \geq 0} = (1, 1, 3, 11, 57, 361, 2763, \ldots)$    
(\href{https://oeis.org/A001586}{A001586}) of the Springer numbers~\cite{Springer71}
bears several combinatorial interpretations; for instance, it counts        
the topological types of odd functions with $2n$ critical values~\cite{Arnold92,Arnold92b}. 
By a result due to Glaisher~\cite{Glaisher1898},
its EGF is $1/(\cos(t) - \sin(t))$, which  is D-algebraic.
As in \S\ref{ssec:tangent} one can prove that 
the corresponding OGF, 
$S(t) = \sum_{n\geq 0} s_n t^n$,  satisfies the difference equation
\[
S \left( {\frac {t}{1 - 2it}} \right) = \left( 2 \, t + i \right) S \left( t \right) + {\frac { \left( 1-i \right)  \left( 2\,t+i \right) }{t+i}},
\]    
and Theorem~\ref{thm:main-intro} implies that $S(t)$
is strongly D-transcendental.

An alternative way to deduce this fact is to use the relations~\cite[\S253]{Glaisher1898} 
\[s_{2n} =  (-1)^n \frac{4^{2n+1}}{4n+2} \cdot U_{2n+1}(1/4), \quad 
 s_{2n-1} =  (-1)^n \frac{4^{2n}}{4n} \cdot U_{2n}(1/4) \]
between the Springer numbers                              
and the values at $x=1/4$ of the Glaisher polynomials $U_n(x)$ in
Table~\ref{Fig:M2}.

\subsubsection{Barred preferential arrangements}

For any $m \in \N$, the sequence $(r_{m, n})_{n \geq 0} = 
(1, m+1, (m+1)(m+3), (m+1)(m^2+8m+13), (m+1)(m^3+15m^2+63m+75), \ldots)$
(\href{https://oeis.org/A226513}{A226513}) 
counting the number of ``barred preferential arrangements'' of $n$ elements with $m$ bars
was studied in~\cite{AhUsPi13}. When $m=0$, this coincides with the sequence of surjection numbers mentioned on page~\pageref{surj-nbs}.
It was shown in~\cite[Thm.~4]{AhUsPi13} that the EGF of $(r_{m, n})_{n\geq 0}$ is equal to $1/(2-e^t)^{m+1}$. 
As in \S\ref{ssec:tangent} one can prove that 
the corresponding OGF, 
$R_m(t) = \sum_{n\geq 0} r_{m,n} t^n$,  satisfies the difference equation
\[
\tau (R_m) = \frac{(m+1)t + 1}{2}R_m + \frac12 
\]
and Theorem~\ref{thm:main-intro} allows to conclude that $R_m(t)$
is strongly D-transcendental.

\subsubsection{Various other sequences}
Many other examples of D-algebraic EGFs, whose corresponding OGFs can be shown to satisfy $\tau$-equations,
may be found in various references such as the book \cite{ErMaObIr1955}
(e.g.\ on page 252 for generalized Bernoulli polynomials;
on page 253 for generalized Euler polynomials; 
on page 254 for generalized Imschenetzky polynomials),
and the articles~\cite{Glaisher1898, 
Glaisher1899, 
Frobenius1910, 
Carlitz59, 
BoBu64, 
KaTh75, 
Luo06, 
OzSiSr10, 
Ozarslan11}, that we will comment in detail in the sequel.
Note however that these references 
never mention explicitly the corresponding
$\tau$-equations for the OGFs.
For all these examples, the OGF is D-transcendental by \cref{thm:ADH},
and even strongly D-transcendental by \cref{thm:main-intro} 
when the order of the $\tau$-equation is $1$.

\paragraph{Karande--Thakare polynomials}
In \cite{KaTh75} Karande and Thakare introduced the family of polynomials $D_n(x; a,k)$, with $a\neq 0$ and $k\in \N$, defined by
\begin{equation} \label{Karande}
\frac{2 \left( \frac{t}{2} \right)^k}{e^t-a} \cdot \exp(xt) = \sum_{n \geq 0} D_n(x; a,k) \frac{t^n}{n!}.
\end{equation}
This is a simultaneous generalization of several families from \cref{Fig:M2}, 
since $B_n(x) = D_n(x;1,1)$,
$E_n(x) = D_n(x;-1,0)$,
$G_n(x) = D_n(x;-1,1)$
and $C_n^{(\lambda )}(x) = \frac{\lambda -1}{2\lambda } D_n(x;1/\lambda ,0)$ for $\lambda  \neq 0$.
Using \cref{prop:diffeqtotau} one can prove that the OGF 
$F(t) \coloneqq \sum_{n \geq 0} D_n(x; a,k) {t^n}$ satisfies the $\tau$-equation
\[
\tau (F) - \frac{t+1}{a} F = - \frac{k! t^k}{2^{k-1} a} \cdot \frac{t+1}{(t+1-tx)^{k+1}}.
\]
The OGF $F(t)$ cannot be D-finite, since otherwise $\sum_{n \geq 0} D_n(x; a,k) \frac{t^n}{n!}$ would be D-finite too, hence  $1/(e^t-a)$ would be D-finite as well by~\eqref{Karande}, a contradiction with the fact that $1/(e^t-a)$ has infinitely many complex singularities. By \cref{thm:main-intro}, $F$ is therefore strongly D-transcendental.

\paragraph{Frobenius--Carlitz rational functions}
Carlitz considered in \cite{Carlitz59} the ``Eulerian'' rational functions $R_n(x)$ defined by
\[
\frac{1-x}{e^t-x} = \sum_{n \geq 0} R_n(x) \frac{t^n}{n!},
\]
and already studied by Frobenius~\cite{Frobenius1910}.
Using \cref{prop:diffeqtotau} one can prove that the OGF 
$F(t,x) \coloneqq \sum_{n \geq 0} R_n(x) {t^n}$ satisfies the $\tau$-equation
\[
\tau (F) - \frac{t+1}{x} F = \frac{x-1}{x}.
\]
Again, by \cref{thm:main-intro}, $F(t,x_0)$ is strongly D-transcendental for any $x_0 \neq 0$.

\paragraph{Generalized Bernoulli, Euler and Carlitz polynomials}
Extensions of the classical Bernoulli, Euler and Carlitz polynomials 
are defined by
(\cite[\S32, p.~185]{Norlund1922}, \cite[\S77, p.~145]{Norlund1924}, \cite[p.~252]{ErMaObIr1953}, \cite[p.~30]{BoBu64})
\begin{align*}
\sum_{n\geq 0}B_n^{(\alpha)}(x) \frac{t^n}{n!}
&\coloneqq \left(\frac{t}{{\rm e}^{t}-1} \right)^{\alpha}
\cdot e^{xt},
\\
\sum_{n\geq 0}E_n^{(\alpha)}(x) \frac{t^n}{n!}
&\coloneqq \left(\frac{2t}{{\rm e}^{t}+1} \right)^{\alpha} \cdot e^{xt},
\\
\sum_{n\geq 0}C_n^{(\lambda , \alpha)}(x) \frac{t^n}{n!}
&\coloneqq \left(\frac{1-\lambda }{1-\lambda {\rm e}^{t}} \right)^{\alpha} \cdot e^{xt} .
\end{align*}
For $\alpha\in\N$, the corresponding OGFs satisfy $\tau$-equations of order $\alpha$, 
hence they are D-transcendental. We conjecture that they are even strongly D-transcendental.
For instance, $F(t) = \sum_{n\geq 0} B_n^{(2)}(x) {t^n}$ satisfies
\[
\tau^2(F) 
=
\frac{4 t +2}{t +1} \cdot \tau(F)
- (2 t + 1) F
-\frac{2 \left(2 t +1\right) t^{2}}{\left(t x -2 t -1\right)^{3}} \, ,
\]
and since $F$ is not D-finite, it is D-transcendental by \cref{thm:ADH}.
Similar conclusions can be drawn for other families of generalized polynomials, such as the 
generalized Apostol-Euler polynomials~\cite{Luo06} and the further unification of the Apostol–Bernoulli, Euler and Genocchi polynomials~\cite{OzSiSr10}, and its generalization to higher orders~\cite{Ozarslan11}.

\paragraph{Glaisher's sequences}
Glaisher introduced in~\cite{Glaisher1898,Glaisher1899} several interesting sequences, that are nowadays 
called Glaisher's $I$, $J$, $H$, $P$, $Q$, $R$ and $T$ numbers.
For instance, the $R$-numbers $1, 7, 305, 33367, \ldots$ (\href{https://oeis.org/A002437}{A002437}) are the coefficients~$R_n$ in the expansion $\cosh(t)/(2\cosh(2t)-1) 
= (1+\cosh(2t))/(2 \cosh(3t))
= \sum_{n \geq 0} (-1)^n R_n t^{2n}/(2n)!$~\cite[\S132, p.~70]{Glaisher1898}.
The corresponding OGF $R(t) \coloneqq \sum_{n \geq 0} (-1)^n R_n t^{2n}$ satisfies the difference equation
\[
R\left(\frac{t}{1+6t} \right) + (6t+1) R(t) -\frac{2 \left(6 t +1\right) \left(7 t^{2}+6 t +1\right)}{\left(5 t +1\right) \left(3 t +1\right) \left(t +1\right)} =0.
\]

Similarly, the $T$-numbers $1, 23, 1681, 257543, \ldots$ (\href{https://oeis.org/A002439}{A002439}) are the coefficients $T_n$ in the expansion $\sinh(t)/(2\cosh(2t)-1) 
= \sinh(2t)/(2 \cosh(3t))
= \sum_{n \geq 0} (-1)^n T_n t^{2n+1}/(2n+1)!$~\cite[\S143, p.~75]{Glaisher1898}.
The corresponding OGF $T(t) \coloneqq \sum_{n \geq 0} (-1)^n T_n t^{2n+1}$ satisfies the difference equation
\[
T\left(\frac{t}{1+6t} \right) + (6t+1) T(t) -\frac{2t \left(6 t +1\right) }{\left(t +1\right) \left(5 t +1\right)}
 =0.
\]
More generally, one can prove using \cref{prop:diffeqtotau} that, for any complex numbers $a,b,c,u,v$, the OGF $F(t)$ associated to the EGF
of $(u \sinh \! \left(a t \right)+v \cosh \! \left(a t \right))/(c \cosh \! \left(b t \right))$
satisfies the difference equation
\[
F\left(\frac{t}{1+2bt} \right) + (2bt+1) F(t) 
+ \frac{2 \left(2 b t +1\right) \left( (a u + b v)t + v \right)}{c \left( (a - b)  t -1\right) 
\left( (a + b) t +1\right)}
 =0.
\]
By \cref{thm:main-intro}, we conclude that $T(t)$, and hence also $\sum_{n \geq 0} T_n t^n$, are (strongly) D-transcendental.

The following is a generalization of Glaisher's  $I$-numbers and $J$-numbers: one starts with 
\[
\sum_{n \geq 0} f_n \frac{t^n}{n!}
\coloneqq
\frac{3}{2} \cdot \frac{{\mathrm e}^{a t}+{\mathrm e}^{-a t}+c}{{\mathrm e}^{b t}+{\mathrm e}^{-b t}+1}
=
1+\frac{c}{2}+\left(a^{2}-\frac{2}{3} b^{2}-\frac{1}{3} b^{2} c \right) t^{2}+\left(b^{4} c +a^{4}-4 b^{2} a^{2}+2 b^{4}\right) t^{4}+\cdots .\]
Then, the corresponding OGF $F(t) = \sum_{n \geq 0} f_n t^n$ satisfies the difference equation
\[
F \left( \frac{t}{1+3bt} \right) -  (3 b t + 1) F(t) + r(t) = 0,
\]
where $r(t)$ is the rational function
\begin{small}
\[
\frac{3 b t \left(3 b t +1\right) \left(\left(a t -b t -1\right) \left(a t +b t +1\right) \left(a t -2 b t -1\right) \left(a t +2 b t +1\right) c +2 \left(b t +1\right) \left(2 b t +1\right) \left(a^{2} t^{2}+2 b^{2} t^{2}+3 b t +1\right)\right)}{2 \left(b t +1\right) \left(2 b t +1\right) \left(a t -b t -1\right) \left(a t +b t +1\right) \left(a t -2 b t -1\right) \left(a t +2 b t +1\right)} .
\]\end{small}
\noindent Glaisher's $I$-numbers~\cite[\S57, p.~35]{Glaisher1898} correspond to the particular choice $(a,b,c) = (0, 1, 0)$ and 
the $J$-numbers~\cite[\S75, p.~44]{Glaisher1898} correspond to the particular choice $(a,b,c) = (1, 2, 0)$.
For any $b\neq 0$, the EGF admits infinitely many singularities, hence it cannot be D-finite; thus $F(t)$ cannot be D-finite either, and by \cref{thm:main-intro} $F(t)$ is necessarily strongly D-transcendental.

\subsubsection{Three more challenging examples}\label{ssec:challenges}
Here we describe three interesting functional equations from the literature where we cannot conclude differential transcendence using the methods of this article.

\paragraph{Fishburn-Stoimenow numbers}
The sequence $(p_n)_{n \geq 0} = (1, 1, 2, 5, 15, 53, 217, 1014, \ldots)$ (\href{https://oeis.org/A022493}{A022493})
with OGF
\[
P(t) = \sum_{n \geq 0} \prod_{i=1}^n \left( 1 - (1-t)^i \right)
\]
is known to count several interesting combinatorial objects.
For instance, Zagier proved in~\cite[Thm.~1]{Zagier01} that $p_n$ is equal to the number of 
certain involutions on $2n$ points, called \emph{regular linearized chord diagrams}: these are involutions $\pi$ in $S_{2n}$ with no fixed points and such that if $\pi_i > \pi_{i+1}$, then $\pi_i > i \geq \pi_{i+1}$.
On the other hand, 
Bousquet-M\'{e}lou et al.~\cite[Thm.~13]{BoClDuKi10} proved that  $p_n$ is the number of unlabeled  posets of size $n$ that do not contain any induced subposet isomorphic to the union of two disjoint 2-element chains (these are called \emph{$(2 + 2)$-free posets}).
Zagier proved~\cite[Thm.~4]{Zagier01} that $p_n/n! \sim \kappa \cdot (6/\pi^2)^n \cdot n^{1/2}$ with $\kappa \approx 2.7$, which implies that $P(t)$ is not D-finite.
He also proved~\cite[Thm.~3]{Zagier01} that 
\[
P(1-e^{-24t}) = e^t \cdot \sum_{n \geq 0} \frac{T_n}{n!} t^n,
\]
where $(T_n)_{n \geq 0} = (1, 23, 1681, 257543, \ldots)$ is the sequence of Glaisher's $T$-numbers.
This implies that $P(t)$ is D-transcendental if and only if the EGF of Glaisher's $T$-numbers is D-transcendental. Note that $\sum_{n \geq 0} T_{n} t^{2n+1}/ (2n+1)!$ is equal to $\sin(2t)/(2 \cos(3t))$, hence it is D-algebraic. Recall that we proved that the OGF $\sum_{n\geq 0} T_n t^{n}
$ is (strongly) D-transcendental. However, none of these results allows to conclude whether the EGF $\sum_{n \geq 0} \frac{T_n}{n!} t^n$ (equivalently, $P(t)$) is D-transcendental or not. We leave this as an open question.
Note that it was proved in \cite[\S6.2]{BoClDuKi10} that $P(t)$ is equal to $F(t,1)$
where $F(t,u) \in \Q[u][[t]] \cap \Q(t)[[u]]$ starts
\[
1+t +\left(u +1\right) t^{2}+\left(u^{2}+3 u +1\right) t^{3}+\left(u^{3}+7 u^{2}+6 u +1\right) t^{4}+  \cdots 
=
\frac{1}{1-t}+\frac{t^{2} u}{\left(1-t \right)^{3}}+\frac{t^{3} \left(1+t -t^{2}\right) u^{2}}{\left(1-t \right)^{6}} + \cdots 
\]
and satisfies the functional equation
\[
F \! \left(t , u\right) 
= 
\frac{\left(1-u \right) \left(1-t \right)}{\left(t u -t +1\right)^{2}}+\frac{u}{\left(t u -t +1\right)^{2}} \cdot F \! \left(t , \frac{u}{t u -t +1}\right).
\]
From the second part of~\cref{thm:ADH}, we can deduce that for any $t_0 \in \C \setminus \{ 0 \}$ with $|t_0|<1$, we have that $F(t_0,u) \in \C[[u]]$ is D-transcendental in~$u$ (this is because the homography $u/(t_0 u -t_0 +1)$ has two fixed points, hence it is conjugated to a dilation).
But, once again, this does not give any information about the nature of $F(t,1) = P(t)$.
%
Another example related to the previous one is the following. 
An unlabeled poset is called \emph{$(3 + 1)$-free} if it does not contain the disjoint union of chains of lengths 3 and 1 as an induced subposet.
The sequence $(q_n)_{n \geq 0} = (1, 2, 5, 15, 49, 173, 639, 2469, \ldots)$ (\href{https://oeis.org/A079146}{A079146}), counting
$(3 + 1)$-free posets with $n$ unlabelled vertices,
was studied by Guay-Paquet, Morales and Rowland in~\cite{GuMoRo14}.
They showed that $q_n \sim 2^{n^2/4- \kappa n \log n + O(n)}$ for some $\kappa>0$, see \cite[Thm.~4.4 and Table 1]{GuMoRo14}, a result which implies that the OGF $\sum_{n \geq 0} q_n t^n$ is not D-finite.
As for $(2 + 2)$-free posets, it is natural to ask whether $\sum_{n \geq 0} q_n t^n$ is D-transcendental, see Question 1 in \href{https://sites.google.com/view/ahmorales/research/conjectures}{Morales' compilation of combinatorial conjectures}.

\paragraph{Schmid's equation}
The following functional equation, with $q\in\N$ and $a,c\in (0,1)$,
was considered in~\cite[Eq.~(2.8)]{Schmid92}:
\[
f(t) = q f(t/q) - \frac{aq}{1+t} \cdot f\left(\frac{ct/q}{1+(1-c)t} \right) + \frac{qt^2}{1+t^2} .
\]
His analysis in Section 3 shows that $f$ is not a rational function.
It would be interesting to prove that $f$ is (strongly) D-transcendental, by using an extension of the methods in this paper.

\paragraph{Generalized digital trees}
Flajolet and Richmond~\cite[Lem.~1]{FlRi92} solved the following difference-differential equation, with 
{$b\in\N$:}
\[
f^{(b)}(t) = 2 e^{t/2} \cdot f\left(\frac{t}{2} \right) + e^t .
\]
Here $f$ is the EGF of the expected number $f_n$ of nonempty nodes in a random tree built from $n$ elements,
{in a tree partitioning process depending on some fixed parameter $b$ (see the introduction of \cite{FlRi92} for more details).}
They showed (\cite[Lem.~2]{FlRi92}) that if $G(t)$ denotes the inverse Borel transform $\fB^{-1}(e^{-t} \cdot f(t))$, then
\[
(1+t)^b \cdot G(t) = 2 t^b \cdot G\left(\frac{t}{2} \right) + t (1+t)^{b-1} .
\]
They proved that $G$ has infinitely many singularities, and that $f_n$ has asymptotic behavior incompatible with D-finiteness. 
Since $b$ is an integer, $G$ is D-transcendental by Theorem~\ref{thm:ADH} with $q=1/2$.
It would be nice to complete the study of the D-algebraic nature of $f$.

\subsection{An elementary treatment of one example: Bernoulli polynomials}  
\label{subsec:Bernoulli}

We give here an elementary proof of the D-transcendence of $B(x,t)$, based on two main steps: first, a
strong link between the expansion of $B(x,t)$ at $t=\infty$ and the Euler gamma function $\Gamma$,
see~\eqref{eq:asymptotic-dev-psi}; second, Hölder's theorem on the D-transcendence of the
$\Gamma$-function~\cite{Holder1887}.

Beyond the case of the OGF $B(x,t)$ of Bernoulli polynomials, {thinking of the}
two-step approach mentioned above, a result of Praagman \cite[Thm.~1]{Praagman} ensures that solutions to general $\tau$-equations are meromorphic at $\infty$; however, there is no hope in general to express such expansions at $\infty$ using special functions, nor to prove that these expansions are D-transcendental. As we shall see in Section~\ref{subsec:difficulties}, Galois theory comes 
{into play} 
to solve this problem, focusing on the equation itself rather than on its solutions.
({Praagman's result}
applies to the meromorphic behavior on $\mathbb C$ of solutions to difference equations with the shift $t\mapsto t+1$, which translates into a similar result on $\mathbb C\cup\{\infty\}\setminus \{0\}$ after our change of variable $t\mapsto \frac{1}{t}$.)

\begin{proposition}\label{prop:exampleBernoulli}
The ordinary generating function $B(x,t)$ 
is D-transcendental over $\C(t)$ for any $x\in\C$.
\end{proposition}

Notice that the
conclusion of Proposition~\ref{prop:exampleBernoulli} (which can alternatively be deduced by combining
Theorem~\ref{thm:ADH} and Lemma~\ref{lem:Bernoulli})\ is weaker than the one
in Theorem~\ref{thm:main-intro}.

\begin{proof}%
We first prove the result for $x=0$, that is     
for the OGF of the sequence of {Bernoulli numbers} $(B_n)_{n \geq 0}$:
\[B_0(t) = B(0,t) =
1-{\frac{1}{2}}t+{\frac{1}{6}}{t}^{2}-{\frac{1}{30}}{t}^{4}+{\frac{1}
{42}}{t}^{6}-{\frac{1}{30}}{t}^{8}+\cdots .
\]                                        

The idea is to use the ``logarithmic Stirling formula'': as $t \to \infty$,
\[ \log \Gamma(t) \sim \left(t-\frac12 \right) \log t - t + \frac{\log (2\pi)}{2} + \sum_{n \geq 1} \frac{B_{2n}}{2n(2n-1)} t^{2n-1}. \]
More precisely, we will use its consequence
on the asymptotic expansion as $t\to \infty$ of the derivative of the digamma function $\Psi(t)\coloneqq  \Gamma'(t)/\Gamma(t)$, see~\cite[Sec.~1.18]{ErMaObIr1953}:
\begin{equation}\label{eq:asymptotic-dev-psi}
    \Psi'(t) \sim_{t\rightarrow \infty} \frac{1}{t}+\frac{1}{2t^2}+\sum_{n=1}^\infty \frac{B_{2n}}{t^{2n+1}},
\end{equation}    
and the fact that $\Gamma(t)$ is D-transcendental.   
More formally, let us introduce the power series in $\C[[t]]$ defined by 
\[
S(t) \coloneqq   t + \frac{t^2}{2} +    \sum_{n=1}^\infty B_{2n}t^{2n+1}.
\]	    
As $\Gamma(t)$ is D-transcendental, it follows that 
$\Psi(t)$ is D-transcendental,
hence the asymptotic expansion \eqref{eq:asymptotic-dev-psi} of $\Psi'(t)$ at infinity is also D-transcendental (if a meromorphic function on $\mathbb C$ is D-algebraic, then its potential asymptotic expansion at infinity satisfies the same differential equations),
and finally $S(t)$ is D-transcendental.      
Since $S(t) = t B_0(t) + t^2$, it also follows that $B_0(t)$ is D-transcendental.

Now, let us treat the case of a general $x\in\C$. The key is the following 
equality, which holds in $\C[x][[t]]$:
\[
  B(x,t) = \frac{1}{t} \cdot S\left( \frac{t}{1+t-tx}\right),
\]
and which readily implies that $B(x,t)$ is D-transcendental.
To prove this last equality, it is sufficient to check that both sides
are power series in $\C[x][[t]]$ that
satisfy the same $\tau$-equation from Lemma~\ref{lem:Bernoulli} (since \eqref{eq:EqBernoulli} admits at most one power series solution).  
This is an easy consequence of the fact that $B_0$ satisfies the equation
$\tau(B_0)  =		
(t+1) \cdot B_0
- \frac {t}{1+t}
$,
hence $S$ satisfies
$\tau(S)  =		
S 
- t^2
$.
\end{proof}

\subsection{Why the situation is not so easy in general}
\label{subsec:difficulties}

We know from \cite[Thm.~1]{Praagman}
that any linear functional equation of the form
\begin{equation*}
   a_0(t)y+a_1(t)\tau(y)+\dots+a_n(t)\tau^n(y)=0, 
\end{equation*}
with $a_0(t),\dots,a_n(t)\in\C(t)$, has a basis of 
meromorphic solutions at $\infty$, i.e., 
in $\C\l(\l\{\frac{1}{t}\r\}\r)$ 
This means that the functional equation \eqref{eq:EqBernoulli}
\[
\tau (B)=(1+t) \cdot B-\frac{t(1+t)}{(1+t-tx)^2}
\]
satisfied by the OGF of the Bernoulli polynomials, has a meromorphic solution at $\infty$. As it turns out, we can find explicitly this solution at $\infty$  with some lucky and elementary manipulations, which might give the impression that the proof of Proposition~\ref{prop:exampleBernoulli} is a bit magical and pulled out of a hat. 

These manipulations are alternatively described in the next lemma.
Once the solution is found, checking its correctness simply amounts 
to using the classical identity
$\Psi'(t)-\Psi'(1+t) = \frac{1}{t^2}$.

\begin{lemma}
The meromorphic function 
$F(x,t)=\frac{1}{t}\Psi'\left(\frac{1+t-tx}{t}\right)$ at $\infty$
is a solution of \eqref{eq:EqBernoulli}, 
for any $x\in\C$. 
\end{lemma}

\begin{proof}
We first remark that $z(t)=1/t$ is solution to the homogeneous equation $\tau (z)=(1+t) \cdot z$. 
Setting $G(x,t)\coloneqq  {F(x,t)}/{z(t)}=tF(x,t)$ yields a simpler functional equation, amenable to telescopic summation:
\begin{equation}
\label{eq:tau_eq_Bernoulli_G-bis}
    \tau (G) = G-\left(\frac{t}{1+t-tx}\right)^2.
\end{equation}
Iterating identity \eqref{eq:tau_eq_Bernoulli_G-bis}, we obtain that $G$ satisfies the following equation for any $n\geq1$ 
\begin{equation*}
    G-\tau^{n}(G)=G(x,t)-G\left(x,\frac{t}{1+nt}\right)=\sum_{k=0}^{n-1}\tau^k \left(\frac{t}{1+t-tx}\right)^2=
 \sum_{k=1}^{n} \left(\frac{t}{1+kt-tx}\right)^2.
\end{equation*}
We remind that {$\Psi'(t)=\sum_{k\geq 0}\frac{1}{(t+k)^2}$.}
Now, letting $n\to\infty$, we conclude that
$G(x,t) =G(x,0)+ \Psi'\left(\frac{1+t-tx}{t}\right)$.
Since {$\lim_{t\to 0^+}\Psi'(1/t)=0$}
we can choose $G(x,0)=0$, hence we have found a solution of the functional equation in Lemma~\ref{lem:Bernoulli} satisfied by $B(x,t)$, namely
 \begin{equation*}
F(x,t)=\frac{1}{t}\Psi'\left(\frac{1+t-tx}{t}\right).
 \end{equation*}
This ends the proof.
\end{proof}
               
The solution $F(x,t)$ is a ``nice'' solution, due to its link with the gamma function. Hölder's theorem implies immediately that 
it is D-transcendental over $\C(t)$, for any $x\in\C$.
One could try to use \eqref{eq:asymptotic-dev-psi} to deduce that 
$B(x,t)$ is the expansion of $F(x,t)$, but this may be a delicate procedure. 
Galois theory comes into the picture to solve this problem. 
One can even say that the solution of such a problem is the core 
of Galois theory, namely recognizing the properties of solutions that depend on the equation, and therefore being able to transfer a property from a solution to another, in spite of the fact that they live in very different 
algebras of functions. For Eq.~\eqref{eq:tau_eq_Bernoulli_G-bis}, 
the question is treated in \cref{exa:Bernoulli}, using \cref{prop:D-transcendecy-rank1-inhomo}, 
where the differential transcendence of $G(x,t)$, or of $t \,B(x,t)$, 
is proven to be equivalent to the differential properties of 
the rational function $\left({t}/{1+t-tx}\right)^2$, that is the inhomogeneous right-hand side of~\eqref{eq:tau_eq_Bernoulli_G-bis}.


\section{Differential transcendence via Galois theory}
\label{sec:DtransGalois}

\subsection{A survey of difference Galois theory}
\label{subsec:Galois}

We quickly review the basic concepts of Galois theory for difference equations.
We follow the classical book~\cite{vdPutSingerDifference}, in order to state the two main theorems of the theory:
the Galois correspondence (see Theorem~\ref{thm:Galois-correspondence}) and the theorem on the dimension of the Galois group
(see Theorem \ref{thm:tr.def=dim}).
All the criteria for the applications we have in mind are their consequences; they form the object of the
subsequent sections.

\medskip
Let us consider a field $\K$ of characteristic zero  equipped with an automorphism
$\tau:\K\to \K$. 
We call $C$, or sometimes $\mathbb K^\tau$, the subfield of $\mathbb K$ of $\tau$-invariant elements of $\mathbb K$, i.e., $C:=\{f\in\mathbb K: \tau(f)=f\}$.
The elements of $C$ are the ``constants of the theory'', therefore they are also called ``$\tau$-constants'' or simply ``constants'', when
the meaning is clear from the context.

\begin{exa}\label{exa:Klazar}
We can take for instance $\K=\C((t))$ and set
$\tau(f(t))\coloneqq  f\l(\frac{t}{t+1}\r)$,
for any $f\in \C((t))$.
We claim that the field of constants $\C((t))^\tau$  of $\C((t))$ coincides with $\C$.
Let us assume that there exists
$\sum_{n\geq -N}a_nt^n\in\C((t))$, for some positive integer $N$, with $a_{-N}\neq 0$,
which is invariant by~$\tau$.
We have
\[
\sum_{n=1}^{N} a_{-n}\l(1+\frac{1}{t}\r)^n+\sum_{n\geq 0}a_n\frac{t^n}{(t+1)^n}=\sum_{n=1}^{N} \frac{a_{-n}}{t^n}+\sum_{n\geq 0}a_nt^n.
\]
By identifying the coefficients of ${t^{1-N}}$, one sees that $Na_{-N}+a_{-N+1}=a_{-N+1}$, hence $a_{-N}=0$.
This is in contradiction with our assumptions and therefore we conclude that we must have $N\leq 0$.
A similar argument allows to exclude the case of a formal power series with $N>0$, 
and to conclude that the only $\tau$-constants
are the actual constants.
Notice that $\tau$ induces an automorphism of $\C(t)$ and of $\C(\{t\})$ as well.
Therefore we also have $\C(t)^\tau=\C(\{t\})^\tau=\C$.
\end{exa}

We consider a linear functional system
$\tau(\vec{y})=A\vec{y}$, where $A$ is an invertible square matrix of order $\nu$ with coefficients in $\K$,
$\vec{y}$ is a vector of unknowns and $\tau$ acts on vectors (and later also on matrices) componentwisely.

\paragraph*{Picard-Vessiot rings}
The Galois theory of difference equations follows the general structure of classical Galois theory. Picard-Vessiot rings play the role of the splitting fields,
where we can find abstract solutions that can be manipulated in the proofs.

\begin{definition}[{\cite[Def.~1.5]{vdPutSingerDifference}}]
A Picard-Vessiot ring for $\tau(\vec{y})=A\vec{y}$ over $\K$ is a $\K$-algebra
$R$ equipped with an automorphism extending the action of $\tau$,
that we still call $\tau$, and such that:
\begin{enumerate}
	\item $R$ does not have any non-trivial proper ideal invariant under $\tau$,
	i.e., $R$ is $\tau$-simple;
	\item there exists $Y\in\GL_\nu(R)$ such that $\tau(Y)=AY$ and
	$R$ is generated by the entries of $Y$ and the inverse of $\det Y$,
	that is $R=\K\l[Y,\det Y^{-1}\r]$. 
\end{enumerate}
\end{definition}

\begin{proposition}[{\cite[\S1.1]{vdPutSingerDifference}}]\label{prop:PVring}
A Picard-Vessiot ring always exists. If $C$ is algebraically closed, then
$R^{\tau}=C$ and $R$ is unique up to an isomorphism of 
$\K$-algebras
commuting with $\tau$.
\end{proposition}

\begin{remark}\label{rmk:PVring}
We will not need the explicit construction of $R$, but it is quite simple and it may be helpful to have it in mind:
one considers the ring of polynomials in the $\nu^2$ variables $X=(x_{i,j})$ with coefficients in $\K$.
Inverting $\det X$ and setting $\tau(X)=AX$, we obtain a ring $\K[X,\det X^{-1}]$ with an automorphism $\tau$.
Any of its quotients by a maximal $\tau$-invariant ideal is a Picard-Vessiot ring of  $\tau(\vec{y})=A\vec{y}$ over $\K$.
\par
It is important to notice that the ring $R$ does not need to be a domain. If $C$ is algebraically closed, 
we can say 
{more about its structure}, namely that
it can be written as a direct sum $R_1\oplus\dots\oplus R_r$, such that:
$R_i=e_i R$, for some $e_i\in R$ with $e_i^2=e_i$;
$R_i$ is a domain; 
$\tau$ acts transitively on the $R_i$'s, i.e., 
changing the order of the $R_i$'s and identifying $\{1,\dots,r\}$ with the elements of $\Z/r\Z$, we have $\tau(R_i)\subset R_{i+1}$.
See \cite[Cor.~1.16]{vdPutSingerDifference}.
\end{remark}

\begin{exa}\label{exa:rank1-homogeneous}
Let $a\in\K$, $a\neq 0$ and $R$ be the Picard-Vessiot ring of $\tau(y)=ay$. 
Then there exists $z\in R$ such that $\tau(z)=az$ and $R=\K[z, z^{-1}]$.
\end{exa}

\begin{exa}\label{exa:rank1inhomogeneous}
Let $a,f$ be non-zero elements of $\K$ and let us consider the functional equation $\tau(y)=ay+f$.
It can be rewritten as
$\tau(\vec{y})=\begin{pmatrix}
                a & f \\
                0 & 1
              \end{pmatrix}\vec{y}$.
This system has two linearly independent solution vectors,
so that an invertible matrix of solutions has the form $Y=\begin{pmatrix}
                                                            z & w \\
                                                            0 & 1
                                                          \end{pmatrix}$,
where $w,z$ are elements of a Picard-Vessiot ring $R$, with $\tau(z)=az$ and $\tau(w)=aw+f$.
Then $R=\K[z, z^{-1},w]$.
\end{exa}

\paragraph*{The Galois group}
We suppose that the field of constants $C$ of $\K$ is algebraically closed.
The Galois group~$G$ of $\tau(\vec{y})=A\vec{y}$ over $\K$ is defined to be the group $\Aut^\tau(R/\K)$
of automorphisms of rings $\varphi:R\to R$ that commute with $\tau$ and such that $\varphi_{\vert \K}$ is the identity map.
\par
Since $\varphi\in G$ leaves $\K$ invariant, the matrix $\varphi(Y)$ is
another invertible matrix of solutions of
$\tau(\vec{y})=A\vec{y}$.
It follows that
$\tau(Y^{-1}\varphi(Y))=Y^{-1}\varphi(Y)$ and hence that $Y^{-1}\varphi(Y)\in\GL_\nu(C)$.
In other words, we have a natural group morphism $G\to\GL_\nu(C)$.
It depends on the choice of $Y$ and a different choice gives a conjugated map.
The theorem below contains two crucial pieces of information:
First of all, $G$ is a geometric object, more precisely it can be identified with the $C$-points of a linear algebraic group.
Roughly, this is another way of saying that $G$ can be identified with a subgroup of matrices of $\GL_\nu(C)$, whose entries
and their determinant are exactly the points in an algebraic variety of
the affine space $\A^{\nu^2+1}_C$.
Secondly, an algebraic relation among the entries of $Y$
exists if and only if the dimension of $G$ is ``smaller than expected''.
These ideas can be formalized as follows:

\begin{theorem}[{\cite[Thm.~1.13 and Cor.~1.18]{vdPutSingerDifference}}]
\label{thm:tr.def=dim}
The morphism $G\to \GL_\nu(C)$, $\varphi\mapsto Y^{-1}\varphi(Y)$,
represents $G$ as the group of the $C$-points of a linear algebraic subgroup of $\GL_\nu(C)$. 
Moreover, the dimension of $G$ over~$C$ as an algebraic variety is
equal
to the transcendence degree of $R$ over $\K$, i.e., $\trdeg_\K R=\dim_C G$.
\end{theorem}

\begin{exa}\label{exa:rank1-homogeneousBIS}
We consider an equation of the form $\tau(y)=ay$, as in Example~\ref{exa:rank1-homogeneous}. 
Its Galois group $G$ is represented through its action on a solution $z\in R$. 
Since any element $\varphi$ of $G$ must send $z$ to another solution of $\tau(y)=ay$,
we have $\varphi(z)=c_\varphi z$, for some $c_\varphi\in C$. 
{As noticed in Example~\ref{exa:rank1-homogeneous}, $R=\K[z,z^{-1}]$, hence $z$ is invertible in $R$, which forces $c_\varphi$ to be non-zero.}
Therefore $G$ is an algebraic subgroup of the multiplicative group $C^*$ of the field $C$. 
The solution $z$ is transcendental over $\K$ if and only if $G=C^*$. 
Since the only proper algebraic subgroups of $C^*$ are the groups of roots of unity, 
$z$ is algebraic over $\K$ if and only if $G$ is a group of roots of unity, i.e., if and only if 
there exists a positive integer $N$ such that $c_\varphi^N=1$, for any 
$\varphi\in G$. Therefore, for any $\varphi\in G$, we have $\varphi(z^N)=c_\varphi^Nz^N=z^N$.
As we will see in Theorem~\ref{thm:Galois-correspondence}, the Galois correspondence implies that $z^N\in \K$.
\par
{For further reference, we point out the following fact. 
As explained in Remark~\ref{rmk:PVring}, the Picard-Vessiot ring $R$ is a quotient of 
a polynomial ring $\K[X,X^{-1}]$ by an ideal invariant under the action of $\tau$.
If $z$ is transcendental over $\K$, such an ideal is necessarily $0$ and $R$ is a domain.}
\end{exa}

\begin{exa}\label{exa:rank1inhomogeneousBIS}
Let us go back to Example \ref{exa:rank1inhomogeneous}. We consider the associated Galois group $G$ over~$C$. Any $\varphi\in G$ must map
a matrix of solutions of the functional equation into another matrix of solutions, therefore there exist
$c_\varphi,d_\varphi\in C$ such that $\varphi(z)=c_\varphi z$ and $\varphi(w)=w+d_\varphi z$.
In other words, we must have:
$\varphi\begin{pmatrix}
          z & w \\
          0 & 1
        \end{pmatrix}=\begin{pmatrix}
          z & w \\
          0 & 1
        \end{pmatrix}\begin{pmatrix}
          c_\varphi & d_\varphi \\
          0 & 1
        \end{pmatrix}$.
Therefore, $G$ is a subgroup
of
$\wtilde G\coloneqq  \l\{\begin{pmatrix}
     c & d \\
     0 & 1
   \end{pmatrix}:c,d\in C, c\neq 0\r\}\subset\GL_2(C)$.
According to whether $z$ and $w$ are algebraically dependent or not, either
$G$ will be a proper linear algebraic subgroup of~$\wtilde G$, or $G=\wtilde G$.
\end{exa}

\paragraph*{The total Picard-Vessiot ring}
We have seen in Remark \ref{rmk:PVring} that $R$ is not a domain, therefore we cannot consider its field of fractions.
However, it is a direct sum of domains $R_i$ 
{and}
we can consider its ring~$\L$ of total fractions, which is the direct sum of
the fields of fractions $\L_i$ of the $R_i$'s.
This means that $\L=\L_1\oplus\dots\oplus \L_r$,
where, for any $i=1,\dots,r$, $\L_i$ is a field and $\tau^r$ induces an automorphism of $\L_i$.
Moreover, $\tau$ acts transitively on the $\L_i$'s, i.e., 
$\tau(\L_i)=\L_{i+1}$, with the same notation as in Remark~\ref{rmk:PVring}.
We will call $\L$ the total Picard-Vessiot ring of $\tau(\vec{y})=A\vec{y}$.

The action of the Galois group $\Aut^\tau(R/\K)$ naturally extends from $R$ to $\L$.
See \cite[\S1.3]{vdPutSingerDifference} for details.
For further reference, we recall the following characterization of total Picard-Vessiot rings.

\begin{proposition}[{\cite[Prop.~1.23]{vdPutSingerDifference}}]
\label{prop:totalPVring}
In the notation above, the total Picard-Vessiot ring $\L$ of $\tau(\vec{y})=A\vec{y}$ is uniquely determined, up to
an isomorphism of $\K$-algebras commuting with $\tau$, by the following properties:
\begin{enumerate}
  \item $\L$ has no nilpotent elements and any non-zero divisor of $\L$ is invertible.
  \item $\L^\tau=C$.
  \item The system $\tau(\vec y)=A\vec y$
  has a matrix solution in $\GL_\nu(\L)$.
  \item $\L$ is minimal with respect to the inclusion and the three previous properties.
\end{enumerate}
In particular, any $\tau$-ring satisfying the first three properties above
contains a copy of the Picard-Vessiot ring~$R$ of $\tau(\vec y)=A\vec y$ over $\K$.
\end{proposition}

\paragraph*{The Galois correspondence}
We consider the set $\cF$ of all $\tau$-stable rings $F\subset\L$, such that $\K\subset F$ and
any element of $F$ is either a zero divisor or a unit in $F$,
and the set $\cG$ of all linear algebraic subgroups of $G$.
For any $H\in\cG$, we set $\L^H=\{f\in \L:\varphi(f)=f \; \text{for all }\varphi\in H\}$ and, for any $F\in\cF$, we set $H_F\coloneqq  \{\varphi\in\cG: \varphi(f)=f
\; \text{for all } f\in F\}$.

\begin{theorem}[Galois correspondence {\cite[Thm.~1.29]{vdPutSingerDifference}}]
\label{thm:Galois-correspondence}
In the notation above, the following two maps are 
each other's inverses:
\begin{equation*}
\begin{array}{ccc}
\cG&\to&\cF\\
H&\mapsto& \L^H
\end{array}
\qquad\hbox{and}\qquad
\begin{array}{ccc}
\cF&\to&\cG\\
F&\mapsto&H_F
\end{array}.
\end{equation*}
In particular, $\L^H=\K$ if and only if $H=G$. 
\end{theorem}

\begin{remark}
{We would like to warn the reader that the normal subgroups of $G$ do not satisfy the 
property that one would expect. See \cite[Cor.~1.30]{vdPutSingerDifference}. 
If $H$ is a normal subgroup of $G$ then the group $\Aut^\tau(\L^H/\K)$ is 
isomorphic to $G/H$. 
Moreover, $H$ is a normal subgroup if and only if for any element 
$f\in\L^H$ there exists $\varphi\in \Aut^\tau(\L^H/\K)$ such that $\varphi(f)\neq f$.}
\end{remark}

We start mentioning the following immediate consequence of the theorem above,
which illustrates how the transcendental nature of a solution to the homogeneous equation transfers to solutions of the inhomogeneous system.
This statement is well known to specialists.

\begin{corollary}
\label{cor:rank1inhomogeneous}
In the notation of Example \ref{exa:rank1inhomogeneousBIS},
if $z$ is transcendental over $\K$ and $w\not\in \K$, then also $w$ is transcendental over $\K$.
\end{corollary}

\begin{proof}
We suppose that $w$ is algebraic over $\K$. Since $w\not\in \K$, there exists $\varphi\in G$ such that
$\varphi(w)=d_\varphi z+w\neq w$, i.e., $d_\varphi\neq 0$. The element $\varphi(w)$ is necessarily algebraic
over $\K$, and therefore $\varphi(w)-w=d_\varphi z$ is also algebraic over $\K$.
\end{proof}

The purpose of the following subsections is to give a proof of an analogous statement for differential
transcendence.


\subsection{Application to differential transcendence}
\label{subsec:DiffTransc}

\paragraph*{The $\partial$-Picard-Vessiot ring}

Now, we suppose that there exists a derivation $\partial$ on $\K$ 
commuting with~$\tau$. 

\begin{exa}\label{exa:Klazar-derivative}
In the situation of Example~\ref{exa:Klazar}, 
for $\tau(f(t))=f\l(\frac{t}{1+t}\r)$, we can take $\partial\coloneqq  t^2\frac{d}{dt}$.
\end{exa}

\begin{proposition}[{\cite{Wibdesc} and
\cite[Prop.~1.16, Rem.~1.18, Cor.~1.19]{diviziohardouinPacific}}]
\label{prop:existence of PV-Wibmer}
For any linear system of the form $\tau(\vec{y})=A\vec{y}$, with $A\in\GL_\nu(\K)$, there exists
a $\K$-algebra $\cR$, equipped with an extension of~$\tau$ and of~$\partial$,
preserving the commutation, such that:
    \begin{enumerate}
    \item there exists $Z \in\GL_\nu(\cR)$ such that $\tau(Z)=AZ$;
    \item $\cR$ is generated over $\K$ by the entries of $Z$, $\frac{1}{\det(Z)}$ and all their derivatives;
    \item $\cR$ is $\tau$-simple.
    \end{enumerate}
Moreover, the total ring of fractions of $\cR$ satisfies the first three properties of Proposition \ref{prop:totalPVring}.
\end{proposition}

We will call the ring $\cR$ the $\partial$-Picard-Vessiot ring of $\tau(\vec{y})=A\vec{y}$ over $\K$, without giving a formal definition.
Applying $\partial^n$ to the system $\tau(\vec{y})=A\vec{y}$ for
any positive integer $n$, we can consider the difference system:
{\begin{equation}\label{eq:prolongation}
\tau(\vec{y})=\l(\begin{array}{cccc}
A & \binom{n}{1}\partial(A) & \cdots & \binom{n}{n}\partial^n(A)\\
0 & A &\ddots  &\vdots  \\
\vdots & \ddots & \ddots & \binom{n}{1}\partial(A) \\
0 & \cdots & 0 & A
\end{array}\r)\vec{y},\quad 
\hbox{with solution~}
\l(\begin{array}{cccc}
Z & \partial(Z) & \cdots & \partial^n(Z)\\
0 & Z &\ddots  &\vdots  \\
\vdots & \ddots~~ & \ddots & \partial(Z)\\
0 & \cdots & 0 & Z
\end{array}\r).
\end{equation}}
The following corollary is hidden in the proof of  
\cite[Prop.~1.16]{diviziohardouinPacific}.
Since $\cR$ is generated by $Z$ and its derivatives, it is a 
fairly natural consequence of Proposition~\ref{prop:totalPVring}. 
{We do not prove it, however we point out that one would mainly need to
prove that the rings generated by $Z$, its derivatives up to a certain order and $Z^{-1}$ are 
$\tau$-simple.}

\begin{corollary}\label{cor:existence of PV-Wibmer}
For any integer $n\geq 0$, the ring $\cR$ contains a copy of the Picard-Vessiot ring of \eqref{eq:prolongation}.
\end{corollary}

Before dealing with the general case of linear $\tau$-difference equations $\tau(y)=ay+f$ of order $1$,
we study the particular cases $f=0$ and $a=1$.

\begin{notation}\label{notation:algebraF}
From now on, $F$ will be a $\K$-algebra with no nilpotent elements, and such that any element is either a zero divisor, or invertible.
We suppose that $F$ is equipped with an extension of $\tau$ and of $\partial$, preserving the commutation, and that $F^\tau=C$.
Of course, it does not need to be the same algebra in all the statements.
\par
Proposition~\ref{prop:totalPVring} thus implies that 
$F$ contains a copy of the (total) Picard-Vessiot ring of
\eqref{eq:prolongation} for any $n\geq 0$. They form an 
ascending chain of subrings of $F$ and their union  
coincides with $\cR$ (see Corollary~\ref{cor:existence of PV-Wibmer}).
\par
We say that an element of $F$ is differentially algebraic over $\K$ if it satisfies an algebraic differential equation (with respect to the derivation $\partial$) with coefficients in $\K$,
and that it is differentially transcendental otherwise.
\end{notation}

\begin{proposition}\label{prop:D-transcendecy-rank1-inhomo}
Let  $f$ be a non-zero element of $\K$, and let $w\in F$ be such that $\tau(w)=w+f$.
Then the following assertions are equivalent:
            \begin{enumerate}
              \item $w$ is differentially algebraic over $\K$.
              \item  There exist a non-negative integer $n$,
            $\a_0,\dots,\a_n\in C$ (not all zero) 
            and $g\in \K$ such that
            $\a_0f+\a_1\partial(f)+\dots+\a_n \partial^n(f)=\tau(g)-g$.
              \item There exist a non-negative integer $n$ and
            $\a_0,\dots,\a_n\in C$ (not all zero)  such that $g\coloneqq  \sum_{i=0}^{n}\a_i\partial^i(w)\in~\!\K$.
            \end{enumerate}
\end{proposition}

\begin{proof}
It follows from Corollary~\ref{cor:existence of PV-Wibmer} that $F$ contains a copy of the Picard-Vessiot ring $R_{f,n}$ of
the system $\l\{\tau(y_i)=y_i+\partial^i(f), i=0,\dots,n\r\}$. 
Notice that 
the latter system
can be written in the form of a linear $\tau$-difference system as follows
(where $\hbox{diagonal}(A_1,A_2,\ldots, A_n)$ indicates a block-diagonal matrix, having the blocks $A_1$, $A_2$, \ldots,~$A_n$ on the diagonal):
\begin{equation}
\tau(\vec{y})=
\hbox{diagonal}\l(\begin{pmatrix}1 & f  \\0 & 1\end{pmatrix},
\begin{pmatrix}1 & \partial(f) \\0 & 1\end{pmatrix},
\dots,
\begin{pmatrix}1 & \partial^n(f) \\0 & 1\end{pmatrix}\r)
\vec{y}.
\end{equation}
For any $i\geq 0$, $\partial^i(w)$ is a solution of $\tau(y_i)=y_i+\partial^i(f)$, so that
$R_{f,n}=\K[w,\partial(w),\dots,\partial^n(w)]\subset F$, as in Example~\ref{exa:rank1inhomogeneous} and Example~\ref{exa:rank1inhomogeneousBIS}.
\par
By definition, the element $w\in F$ is differentially algebraic over $\K$ if and only if there exists
$n\geq 0$ such that $w,\partial(w),\dots,\partial^n(w)$ are algebraically dependent over $\K$, therefore if and only if there
exists $n\geq 0$ such that $\trdeg_\K R_{f,n}\leq n$.
Moreover, for any $\varphi\in G_n\coloneqq  \Aut^\tau(R_{f,n}/\K)$,
we must have $\varphi(\partial^i(w))=\partial^i(w)+d_{\varphi, i}$, for some $d_{\varphi,i}\in C$.
The composition of two automorphisms $\varphi$ and $\psi$
in the Galois group is represented by the sum  $d_{\varphi, i}+d_{\psi, i}$.
This means that we can identify the Galois group to a subgroup of the vector space $(C,+)^{n+1}$.
Theorem \ref{thm:tr.def=dim} implies that $\trdeg_\K R_{f,n}\leq n$
if and only if $\dim_C G_n\leq n$, hence if and only if 
$G_n$ is a proper linear subgroup of $(C,+)^{n+1}$, 
i.e., $G_n$ is contained in a hyperplane of $C^{n+1}$.
Thus, the property $\trdeg_\K R_{f,n}\leq n$ is equivalent to the existence of  $\a_0,\dots,\a_n\in C$ (not all zero) such that
for any $\varphi\in G$, we have $\sum_{i=0}^{n}\a_id_{\varphi,i}=0$.
The last linear relation is equivalent to the fact that $g\coloneqq  \sum_{i=0}^{n}\a_i \partial^i(w)$ is $G$-invariant,
that is $\varphi(\sum_{i=0}^{n}\a_i \partial^i(w))=\sum_{i=0}^{n}\a_i \partial^i(w)$ for any $\varphi\in G$.
Theorem \ref{thm:Galois-correspondence} implies  that the latter condition is equivalent to the fact that $g$ belongs to $\K$
and we obtain:
\begin{equation*}
    \tau(g)-g=\tau\l(\sum_{i=0}^{n}\a_i \partial^i(w)\r)-\sum_{i=0}^{n}\a_i \partial^i(w)=\sum_{i=0}^{n}\a_i \partial^i(f).
\end{equation*}
 The proof is completed.
\end{proof}

The equivalence between the first two assertions of the previous statement is 
an immediate corollary 
of \cite[Prop.~2.1]{Hardouin-Compositio}, applied to the system 
$\l\{\tau(y_i)=y_i+\partial^i(f), i=0,\dots,n\r\}$.
The equivalence with the third assertion 
is contained in the cited proof and explains how, under these assumptions, 
differential algebraicity (i.e., existence of an algebraic differential equation for $w$) 
and D-finiteness (i.e., existence of a \emph{linear} differential equation for $w$) are equivalent. 
It also explains how to recover a 
linear differential equation for the solution~$w$ from the ``telescoper'' $\sum_{i=0}^n a_i \partial^i$ for~$f$ in the second assertion. 
One can find it in this form in \cite[Cor.~6.7]{divizio-hanoi}. 
Let us notice that a similar criterion, but in a much more sophisticated setting, has been proved by Papanikolas in 
\cite{Papanikolas-Inventiones} and applied to the algebraic independence of 
Carlitz logarithms (see Thm.~1.2.6 in \emph{loc.~cit.}).  
Even though they are defined in positive characteristic, 
they present some analogies with the ordinary and exponential generating functions considered in \S\ref{sec:setting}.

\begin{exa}\label{exa:Bernoulli}
We consider the OGF of the family of Bernoulli polynomials,
which satisfies the 
functional equation 
\[\tau (B)=(1+t) \cdot B-\frac{t(1+t)}{(1+t-tx)^2},\]
as proved in Lemma~\ref{lem:Bernoulli}.
Setting $G(x,t)\coloneqq  tB(x,t)$, one obtains the functional equation:
\[
    \tau (G) = G-\left(\frac{t}{1+t-tx}\right)^2 .
\]
We need to prove that $G$ is differentially transcendental over $\C(t)$, for any fixed value of $x\in\C$, to conclude the differential transcendence of $B$. 
We do so by supposing by contradiction that $G$ is differentially algebraic over $\C(t)$.
The argument that follows will appear again in the proof of Theorem~\ref{thm:ADH++} 
(precisely in the proof of Lemma~\ref{lemma:inhomogenouus-ADH++}) 
and is actually used frequently in this kind of problems.
One can apply Proposition~\ref{prop:D-transcendecy-rank1-inhomo}
and show that for any non-negative integer $n$, there do not exist any constants $\a_0,\ldots, \a_n$ (not all zero) and any function $g\in\C(t)$ such that
\begin{equation}
\label{eq:Bernoulli_tel-1}
    \a_0f+\a_1\partial(f)+\dots+\a_n \partial^n(f)=\tau(g)-g,
\end{equation}
with $\partial =  t^2\frac{d}{dt}$ and $f=\left(\frac{t}{1+t-tx}\right)^2$. We easily compute, for all $k\geq 0$,
\begin{equation}
\label{eq:Bernoulli_tel-2}
    \partial^{k}(f) = (k+1)!\left(\frac{t}{1+t-tx}\right)^{k+2}.
\end{equation}
Let us first deal with the case $x\neq 1$. 
By \eqref{eq:Bernoulli_tel-2}, the left-hand side of 
\eqref{eq:Bernoulli_tel-1} has a unique pole, namely, 
at $t_0$ such that $1+t_0-t_0x=0$ (i.e., $t_0=\frac{1}{x-1}$). 
This proves that the right-hand side of \eqref{eq:Bernoulli_tel-1} 
has only one pole at $t_0$, 
and hence either $g$ or $\tau(g)$ has a pole at $t_0$.
To fix ideas, let us assume that $g$ has a pole at $t_0$. Then $\tau^{-1}(t_0)$ is a pole of $\tau(g)$. If $\tau(g)-g$
has a pole at $\tau^{-1}(t_0)$, we have proved that $\tau(g)-g$ has at least two poles, 
a contradiction with~\eqref{eq:Bernoulli_tel-1}. 
If $\tau(g)-g$ has no pole at $\tau^{-1}(t_0)$, then we conclude that $g$ has also a pole at $\tau^{-1}(t_0)$, that cancels the pole of $\tau(g)$
in $\tau(g)-g$. It means that $\tau(g)$ must have a pole at $\tau^{-2}(t_0)$ and we can repeat the argument. 
Since $g$ and $\tau(g)$ are rational functions, they admit a finite number of poles, therefore eventually one shows that 
$\tau(g)-g$ has two poles, contradicting the fact that the left-hand side of \eqref{eq:Bernoulli_tel-1} has only one pole.
\par
In the case $x=1$, the left-hand side of \eqref{eq:Bernoulli_tel-1} becomes a non-zero polynomial with no constant term. 
One first proves that $g\in\C(t)$ must also be a polynomial and finally gets a contradiction, 
showing that it can only be a constant, in contradiction with the fact that the left-hand side is non-zero.
{The key point here is that $\infty$ is the only 
pole of a polynomial and that $\tau$ does not fix $\infty$.}
\end{exa}

Let $a$ be a non-zero element of $\K$. We consider the equation $\tau(y)=ay$ and the Picard-Vessiot ring $R_{a,n}$  over $\K$
of the linear difference system obtained as in \eqref{eq:prolongation} for $A=(a)$.
The following proposition is proved by using the trick of taking logarithmic derivatives, 
as in \cite[Prop.~2.2]{Hardouin-Compositio}.

\begin{proposition}\label{prop:rank1homogeneous-telescopers}
Let $a\in\K$ be non-zero, $F$ be a $\K$-algebra as in Notation \ref{notation:algebraF}
and let $z\in F$ be a non-zero solution of $\tau(y)=ay$.
With the notation introduced above, the following assertions are equivalent:
\begin{enumerate}
  \item The element $z$ is differentially algebraic over $\K$.
  \item There exists a non-negative integer $n$ such that $\trdeg_\K R_{a,n}\leq n$.
  \item There exists a non-negative integer $n$ such that  there exist $\a_0,\dots,\a_n\in C$ (not all zero) and $g\in \K$
satisfying
\begin{equation}\label{eq:differential-telescoper-homogeneous}
\a_0\frac{\partial(a)}{a}+\a_1\partial\l(\frac{\partial(a)}{a}\r)+\dots+\a_n\partial^{n}\l(\frac{\partial(a)}{a}\r)=\tau(g)-g.
\end{equation}
\item  There exists a non-negative integer $n$  and $\alpha_0,\ldots, \alpha_n \in C$ (not all zero) such that $g := \sum_{i=0}^n \alpha_i \partial^i( \partial(z)/z) \in \mathbb{K}$.
\end{enumerate}
\end{proposition}

\begin{proof}
We have $R_{a,n}=\K[z,\partial(z),\dots,\partial^{n}(z), z^{-1}]\subset F$, up to an automorphism commuting with~$\tau$.
The first equivalence follows immediately from the definition of differential algebraicity and
Theorem \ref{thm:tr.def=dim}.
We notice that $z$ is differentially algebraic over $\K$ if and only if $w\coloneqq  \frac{\partial(z)}{z}\in F$
is differentially algebraic over $\K$.
Since we have 
$\tau(w)
=
\frac{\tau(\partial(z))}{\tau(z)}
=
\frac{\partial(\tau(z))}{\tau(z)}
=
\frac{\partial(az)}{az}
=
w+\frac{\partial(a)}{a}$, 
we conclude by applying Proposition~\ref{prop:D-transcendecy-rank1-inhomo}.
\end{proof}

\begin{exa}\label{exa:Bell-homogeneous}
Let us consider the setting of Examples \ref{exa:Klazar} and \ref{exa:Klazar-derivative}, with
$\tau(f(t))=f\l(\frac{t}{1+t}\r)$ and 
$\partial=t^2\frac{d}{dt}$.
The function $z(t)\coloneqq  \Ga\l(\frac{1}{t}\r)^{-1}$, 
which lives in the algebra
of analytic functions over $\C^*\cup\{\infty\}$,
is a solution of the equation $\tau(y)=ty$, that is 
the homogeneous equation associated to Klazar's example~\eqref{eq:Bell-intro}. 
Since $\frac{\partial(t)}{t} = t$
and
$\partial^n(t) = n! t^{n+1}$,
one deduces that 
for any $\alpha_0, \ldots, \a_n\in \C$ (not all zero),
the linear combination
\begin{equation*}
\a_0t+\a_1\partial\l(t\r)+\dots+\a_n\partial^{n}\l(t\r)
=
\a_0t+ \a_1 t^2+\cdots+\a_n \, n! t^{n+1}
\end{equation*}
cannot be written as $g\l(\frac{t}{1+t}\r) - g(t)$, for any $g\in \C(t)$. This is easily  proved by reasoning on the poles of the potential rational functions~$g$,
as in Example~\ref{exa:Bernoulli}.
We deduce from Proposition~\ref{prop:rank1homogeneous-telescopers} that 
$z(t)\coloneqq  \Ga\l(\frac{1}{t}\r)^{-1}$ is D-transcendental over $\K\coloneqq  \C(t)$, 
which reproves Hölder's {theorem~\cite{Holder1887,Bank-Kaufman-1978-Holder,HardouinSinger}}
on the D-transcendence of the gamma function.
We also deduce from Proposition~\ref{prop:rank1homogeneous-telescopers} that any non-zero solution of
$\tau(y)=ty$ in any algebra $F$ as above is differentially transcendental\footnote{After we submitted to the arXiv a \href{https://arxiv.org/abs/2012.15292v1}{first version} of the present paper, Michael F.\ Singer has brought to our attention the fact that in 2010, in a private letter to Charlotte Hardouin, he had given a Galoisian proof of Klazar's theorem on the OGF of Bell numbers, based on some calculations
very similar to this example and 
\cite[Prop.~3.9]{HardouinSinger}.} over $\C(t)$.
\end{exa}

\paragraph*{Differential transcendence: the equation $\tau(y)=ay+f$}

We want to prove a generalization of Corollary \ref{cor:rank1inhomogeneous} to differential
transcendence (see Theorem~\ref{thm:mainTHM} below).
Let $f$ and $a$ be non-zero elements of $\K$, and
let us consider the difference equation
$\tau(y)=ay+f$.
It can be transformed in a linear system as in the examples above:
$\tau(\vec{y})=A\vec{y}$, with $A\coloneqq  \begin{pmatrix}a & f \\0 & 1\end{pmatrix}$.
The following statement generalizes 
\cite[Item 1 of Prop.~3.8]{HardouinSinger} to the case of an algebraically closed field of constants.
We remind that in \cite{HardouinSinger} the authors assume that the field of constants is 
differentially closed (although it is not difficult to prove a descent 
to any algebraically closed field, for the readers familiar with this kind of reasoning).
The advantage of the statement below, compared to \cite{HardouinSinger}, 
is that one can look for the solutions of the inhomogeneous equation and 
the associated homogeneous equation in two different algebras.
As we have already observed in \S\ref{subsec:difficulties}, this is particularly useful 
in our setting.
We will illustrate further the situation in Remark~\ref{rmk:Gamma} below. 

\begin{theorem}\label{thm:mainTHM}
Let us consider an equation of the form $\tau(y)=ay+f$, with $a,f\in \K$, such that $a\neq 0,1$ and $f\neq 0$.
Let $F/\K$ be a field extension, 
{with an extension of $\tau$ and $\partial$ to $F$ preserving the
commutativity},
such that there exists $w\in F\setminus \K$
satisfying the equation $\tau(w)=aw+f$.
Moreover, let $F_a$ be a $\K$-algebra as in Notation \ref{notation:algebraF},
such that there exists $z\in F_a$ satisfying the equation $\tau(z)=az$.
If $z$ is differentially transcendental over $\K$, then $w$ is differentially transcendental over $\K$.
\end{theorem}

Proposition \ref{prop:rank1homogeneous-telescopers} and Theorem \ref{thm:mainTHM}
imply directly the following corollary:

\begin{corollary}\label{cor:mainTHM}
In the notation of the theorem above,
if for any $n\geq 0$, any $\a_0,\dots,\a_n\in C$ (not all zero) and any $g\in \K$,
we have
\begin{equation}
\a_0\frac{\partial(a)}{a}+\a_1\partial\l(\frac{\partial(a)}{a}\r)+\dots+\a_n\partial^{n}\l(\frac{\partial(a)}{a}\r)\neq \tau(g)-g,
\end{equation}
then $w$ is differentially transcendental over $\K$.
\end{corollary}

\begin{remark}\label{rmk:Gamma}
Corollary~\ref{cor:mainTHM} allows us to conclude immediately that the generating function of the Bell numbers is differentially transcendental over $\C(t)$,
as a consequence of Hölder's theorem. See Example \ref{exa:Bell-homogeneous}.
If we want to apply Theorem \ref{thm:mainTHM}, it is enough to take $F=\C((t))$ and $F_a$  the field of meromorphic functions at $\infty$.
Notice that there is no common natural extension of the fields $F$ and $F_a$. 
\end{remark}

We will rather prove the following statement, which is obviously equivalent to Theorem~\ref{thm:mainTHM}:

\begin{proposition}\label{prop:D-transcendence-Klazar}
In the notation of Theorem \ref{thm:mainTHM}, if $w\in F\setminus \K$ is differentially algebraic over~$\K$, then
$z\in F_a$ is differentially algebraic over $\K$.
\end{proposition}

\begin{proof}
Let $R_n$ be the Picard-Vessiot ring over $\K$ of the system obtained from $\tau(\vec{y})=
    \begin{pmatrix}a & f \\0 & 1\end{pmatrix}
    \vec{y}$ by $n$-fold derivation
    as in \eqref{eq:prolongation},
and $\cR$ be its $\partial$-Picard-Vessiot ring over $F$.
It follows from Proposition~\ref{prop:existence of PV-Wibmer} and Proposition~\ref{prop:totalPVring}
that we can suppose without loss of generality that $R_n\subset \cR$, because $\K\subset F$, by assumption.
Moreover, 
\begin{equation*}
   R_n=\K[\tilde z,\partial(\tilde z),\dots,\partial^n(\tilde z),w,\partial(w),\dots,\partial^n(w),\tilde{z}^{-1}],
\end{equation*}
for some $\tilde z$ satisfying $\tau(\tilde z)=a\tilde z$.

Let us assume that $w$ is differentially algebraic over $\K$.
This means that $w$ satisfies an algebraic differential equation 
of order $\kappa$ with coefficients in $\K$.
Let $\varphi\in\Aut^\tau(R_\kappa/\K)$ and let us consider $\varphi(w)\in R_\kappa$
such that $\tau(\varphi(w))=a\varphi(w)+f$. The derivatives of $\varphi(w)$ may not belong to 
$R_\kappa$, since $\varphi$ commutes only with $\tau$ and not necessarily with $\partial$, 
however the $\K$-algebra 
\begin{equation*}
   \widetilde R_\kappa=\K[\tilde z,\partial(\tilde z),\dots,\partial^\kappa(\tilde z),\varphi(w),\partial(\varphi(w)),\dots,\partial^\kappa(\varphi(w)),z^{-1}]\subset\cR
\end{equation*}
is another Picard-Vessiot ring of the difference system obtained by $\kappa$-fold iteration. 
Hence $R_\kappa$ and $\widetilde R_\kappa$ are isomorphic as {difference $\K$-algebras
(meaning that the isomorphism of $\K$-algebras commutes with $\tau$)}
and, in particular, they have the same transcendence degree over $\K$. 
We deduce that $\varphi(w)\in\cR$ is solution of an algebraic differential equation of order $\kappa$,
with coefficients in $\K$. Finally, $w-\varphi(w)$ satisfies an algebraic differential equation of order at most $2(\kappa+1)$ (see for instance~\cite[Thm.~2.2]{BoRu86}).
Since $\bar z\coloneqq  w-\varphi(w)\in\cR$ is a solution of $\tau(y)=ay$, we deduce by Proposition~\ref{prop:rank1homogeneous-telescopers},
applied to $\bar z\in\cR$, that there exist
$\a_0,\dots,\a_n\in C$ (not all zero) and $g\in\K$  such that $\sum_{i=0}^{n}\a_i\partial^i\l(\frac{\partial(a)}{a}\r)=\tau(g)-g$.
We conclude, by applying again Proposition~\ref{prop:rank1homogeneous-telescopers} to $z\in F_a$,   
that $z$ is itself differentially algebraic over $\K$.
\end{proof}

\begin{remark}
{The previous proof would be easier if we were able to find an automorphism 
$\varphi$ which commutes with both $\tau$ and $\partial$. 
This is the starting point of the paper \cite{HardouinSinger}, where a more sophisticated 
Galois theory to deal with differential transcendence is developed. 
The idea is to consider as Galois group the group of automorphisms of the 
$\partial$-Picard-Vessiot ring introduced in Proposition~\ref{prop:existence of PV-Wibmer}, 
which commutes with $\tau$ \emph{and} $\partial$. 
Unfortunately, under the assumption that $C$ is algebraically closed, such a group can be reduced to the identity.
To avoid this problem one has either to enlarge drastically $C$, as in \cite{HardouinSinger}, 
or to adopt a group scheme point of view, as for instance in \cite{diviziohardouinPacific}.
Here, we have decided to avoid 
more abstract theories, since we only deal with order $1$, which can be handled more easily.}
\end{remark}

\section{Main result: Strong differential transcendence 
for first-order difference equations}
\label{sec:main}

We consider a field $\K_0\coloneqq  C(t)$, where $C$ is an algebraically closed field of characteristic zero, equipped with an
Archimedean norm $|\cdot|$.
Typically, we will choose $C$ to be the algebraic closure $\ol\Q$ of $\Q$ or $\C$, with the usual norm.
\par
We will look for solutions in $\F\coloneqq  C((t))$ and we will establish their differential transcendence over the field $\K\coloneqq  C(\{t\})$,
where the convergence is considered with respect to $|\cdot|$.
We {recall}
(\cref{exa:Klazar-derivative}) that the derivation $\partial\coloneqq  t^2\frac{d}{dt}$ commutes with $\tau:f(t)\mapsto f(t/(1+t))$ and that establishing the
differential transcendence with respect to $\partial$ is equivalent to establishing it with respect to $\frac{d}{dt}$.
\par
The next theorem generalizes \cite[Thm.~2]{Nishioka_note_1984}, where the author proves a similar statement 
for the differential transcendence over $\K_0$. 
It should be the
first step of a generalization of \cite[Thm.~1.2]{Adamczewski-Dreyfus-Hardouin} 
to higher-order difference equations (in the case called $S_\infty$ in \emph{loc.~cit}).

\begin{theorem}\label{thm:ADH++}
Let $a,f\in\K_0$, with 
$a\neq 0$, 
and let $w\in\F\setminus\K_0$ satisfy the difference equation $\tau(w)=aw+f$.
Then $w$ is differentially transcendental over $\K$.
\end{theorem}

A simple rational change of variable shows the following corollary: 

\begin{corollary}\label{cor:ADH++}
The theorem above holds if we replace $\tau$ with the endomorphism associated 
to any homography with only one fixed point $t_0$, 
$\K$ with the field of germs of meromorphic functions at $t_0$ 
and $F$ with the field of 
formal Laurent series based at $t_0$. 
\end{corollary}

Theorem~\ref{thm:ADH++} above implicitly says that  
any solution $w$ of $\tau(w)=aw+f$, meromorphic in a punctured neighborhood of $0$, is the Taylor expansion of a rational function. 
We start by proving this fact:

\begin{lemma}\label{lemma:exclude-K}
Let $a,f\in\K_0$ with $a\neq 0$,
and let $w\in\K$ satisfy the difference equation $\tau(w)=aw+f$.
Then $w\in\K_0$.
\end{lemma}

\begin{proof}
Notice that if $\tau(w)=aw+f$, then $\tau^2(w)=\tau(a)aw+\l(\tau(a)f+\tau(f)\r)$, with 
$\tau(a)a,\tau(a)f+\tau(f)\in\K_0$. By induction, one proves that 
for any $n\geq 2$, $w$ satisfies a functional equation of the form $\tau^n(w)=a_nw+f_n$, 
with $a_n\coloneqq  \tau^{n-1}(a)\cdots\tau(a)a\in\K_0$ and $f_n\in\K_0$
{defined recursively by $f_n=\tau(a_{n-1})f+\tau(f_{n-1})$}
\par
By assumption, $w$ is analytic in an open punctured disk of radius $r$ centered at $0$.
Let $t_0\in C\setminus\{-\frac{1}{k},~k\in\Z_{\geq 1}\}$ be outside such a disk, i.e., $|t_0|\geq r$.
Then $\tau^n(t_0)=\frac{t_0}{1+nt_0}$, for any $n\geq 1$. Therefore 
{$\tau^n(t_0)$} 
tends to $0$ as $n$ tends to $+\infty$. 
This means that there exists $n$ such that $|\tau^n(t_0)|<r$, so that the functional equation 
$\tau^n(w)=a_nw+f_n$ allows to continue $w$ to a meromorphic function in a neighborhood of $t_0$. 
If $t_0=-\frac{1}{k}$, for some positive integer $k$, then $\tau^{-n}\left(-\frac{1}{k}\right)=-\frac{1}{n+k}$, which also tends to $0$. 
Applying $\tau^{-1}$ recursively to $\tau(w)=aw+f$, 
we obtain a functional equation that allows to continue $w$ in a neighborhood of $t_0$. 
Summarizing, $w$ is actually meromorphic on the 
{whole of $C$}. 
In particular $w$ is meromorphic in a neighborhood of $-1$, with $\tau(-1)=\infty$, 
therefore the initial functional equation $\tau(w)=aw+f$ allows to continue 
$w$ to a meromorphic function in a neighborhood of $\infty$. 
Finally $w$ can be continued to a meromorphic function on $C\cup\{\infty\}$, 
hence it is actually a rational function. 
\end{proof}

The following corollary explains the strategy of the proof:

\begin{corollary}\label{cor:exclude-K}
Theorem~\ref{thm:ADH++} is equivalent to the following statement:
\begin{enumerate}[label=({\Alph{*}}),ref=(\textit{\Alph{*}})]\setcounter{enumi}{19}
    \item\label{property-T}Let $a,f\in\K_0$ with $a\neq 0$,
and let $w\in\F\setminus\K$ satisfy the difference equation $\tau(w)=aw+f$.
Then $w$ is differentially transcendental over $\K$.   
\end{enumerate}
\end{corollary}

\begin{proof}
If $w$ satisfies the hypotheses of \ref{property-T}, then Theorem~\ref{thm:ADH++} implies that $w$ is 
differentially transcendental over $\K$. Therefore \ref{property-T} is a consequence of Theorem~\ref{thm:ADH++}.
\par 
On the other hand, let $w\in\F\setminus\K_0$ as in Theorem~\ref{thm:ADH++}. If $w\in\K$, then $w\in\K_0$, thanks to Lemma~\ref{lemma:exclude-K}, 
which contradicts the hypotheses of Theorem~\ref{thm:ADH++}, hence $w\in\F\setminus\K$.
Statement \ref{property-T} implies that $w$ is differentially transcendental over $\K$.
This proves that \ref{property-T} implies Theorem~\ref{thm:ADH++}. 
\end{proof}

Thanks to the statement \ref{property-T}, the strategy of the proof is becoming clear: if the equation
$\tau(y)=ay$ has a differentially transcendental solution, one applies Theorem~\ref{thm:mainTHM};
otherwise we need to look closer to the structure of the
space of solutions of $\tau(y)=ay$.
The following couple of lemmas are the main steps in the proof of Theorem~\ref{thm:ADH++}.

\begin{lemma}
\label{lemma:constat-a-PV}
Let 
{$a\in C\setminus\{0,1,\hbox{roots of unity}\}$}
and let $F_a$ be a $\K$-algebra as in Notation
\ref{notation:algebraF} containing a solution $z$ of $\tau(y)=ay$, minimal for the inclusion.
Then $z$ satisfies a differential equation of the form $\partial(z)=cz$, 
{for some $c\in C$},
and $F_a$ is a field of the form $F_a=\K(z)$, 
with $z$ transcendental over $\K$.
\end{lemma}

\begin{remark}
In Lemma~\ref{lemma:constat-a-PV} 
we may of course take $F_a=\F$,
but such a result would not be enough to prove Theorem~\ref{thm:ADH++}. 
Indeed, generically one cannot expect that both the functional equations $\tau(w)=aw+b$ and $\tau(z)=az$   
have a solution in $\F$. Most examples, such as the generating function of the Bell numbers, actually do not.
\par
We point out that if $a=1$, then $F_a=\K$, since $z\in C$. 
\end{remark}

\begin{proof}[Proof of Lemma~\ref{lemma:constat-a-PV}.]
{It follows from}
Proposition~\ref{prop:totalPVring} that $F_a$ is generated by $z$ over $\K$ and that $F_a$ 
is unique up to a $\K$-algebra isomorphism commuting with $\tau$. 
Since $\partial$ and $\tau$ commute, we have  $\tau(\partial(z))=a\partial(z)$. Therefore $z$ 
and $\partial(z)$ are solutions to the same functional equation and $\partial(z)/z=c\in C$.
\par
{We want to show that $z$ is transcendental over $\K$, by contradiction. 
Following Example~\ref{exa:rank1-homogeneousBIS}, if $z$ is algebraic over $\K$ then 
there exists a positive integer $N$ such that $z^N\in\K$. It follows from 
Lemma~\ref{lemma:exclude-K} that we actually have $z^N\in\K_0$, with $\tau(z^N)=a^Nz^N$.
Since $a\not\in\{0,1,\hbox{roots of unity}\}$, iterating the functional equation $\tau(z^N)=a^Nz^N$ one shows that
if $z^N$ has a non-zero pole (resp.\ zero), then 
it has an infinite number of poles (resp.\ zeros).
Therefore $z^N$ has a pole or a zero only at $0$ and it must be a monomial. 
If $z^N$ is a monomial, then $a=\tau(z^N)z^{-N}$ is an integer power of $(1+t)$ and cannot be a constant, unless 
$z^N$ is a constant itself, i.e., $N=0$ and $a=1$. We have found a contradiction, hence $z$ is transcendental.
Thanks to Example~\ref{exa:rank1-homogeneousBIS}, 
we conclude that the Picard-Vessiot ring of $\tau(z)=az$ is a domain and hence that $F_a=\K(z)$ is a field.}
\end{proof}

\begin{exa}\label{exa:exponentials}
{Let $a\in C\setminus\{0,1,\hbox{roots of unity}\}$.
Let us fix a branch of the complex logarithm, so that 
$\exp\left(-\frac{\log a}{t}\right)$ is a solution to $\tau(y)=ay$. 
The field 
$\K\left(\exp\left(-\frac{\log a}{t}\right)\right)$  is a total Picard-Vessiot ring for $\tau(y)=ay$ over $\K$.
Notice that $\exp\left(-\frac{\log a}{t}\right)$ is not meromorphic at $0$, 
hence its expansion does not belong to $\F$.
On can show by a direct substitution that in this case $\tau(y)=ay$ cannot have a solution in $\F$.
In other words, the hypothesis of Theorem~\ref{thm:ADH++} with $f=0$ is never satisfied.}
\end{exa}

The proof of the following lemma is an elaborated version of the argument of \cref{exa:Bernoulli}:

 \begin{lemma}\label{lemma:inhomogenouus-ADH++}
 Let $a\in C$, $f\in \K_0$, $a\neq 0$, and let $w\in\F\setminus \K_0$ satisfy $\tau(w)=a w+f$.
 Then $w$ is differentially transcendental over $\K$.
 \end{lemma}

\begin{proof}
Before starting the actual proof we show that we can make some extra assumptions without loss of generality:
\begin{enumerate}
    \item\label{itone}
{If $a$ is a root of unity, we can assume that $a=1$. In fact, if $a$ is a root of unity of order $N$, 
we can proceed as in the proof of Lemma~\ref{lemma:exclude-K} and consider the iterated functional equation
$\tau^N(w)=a^Nw+f_N=w+f_N$. 
Recall that $\tau^N(t)=\frac{t}{1+Nt}$, by~\eqref{eq:iter}.
If we set $\widetilde t:=Nt$, then $\tau^N(\widetilde t)=\frac{\widetilde t}{1+\widetilde t}$, 
therefore we have an equation of the form $\tau(w)=w+f_N$ in the variable $\widetilde t$. 
Notice that neither the iteration of the initial functional equation nor the change of variable 
that we have performed change the nature of the solutions $w$.}
    \item\label{ittwo}
We may assume that $f$ does not have a pole at $\infty$. Indeed, we can replace $w$ by $\tau^n(w)$, which satisfies the 
functional equation $\tau(\tau^n(w))=a\tau^n(w)+\tau^n(f)$.
If $f$ has a pole at $\infty$, then $\tau^n(f)=f\left(\frac{t}{1+nt}\r)$ has a pole at $t=-1/n$, and by  conveniently choosing~$n$, it does not have a pole at $\infty$. 
    \item\label{itthree}
{Finally, we can assume that $f$ has only one pole in each $\tau$-orbit (i.e., in each set of the form $\{\tau^n(\alpha):n\in \mathbb Z\}$, for $\alpha\in C$), that does not contain~$0$. 
}
Indeed, notice that we can replace $w$ by $w+r$ for any $r\in\K_0$, and change the equation accordingly, 
since
this will change neither the hypotheses  nor the conclusion of the lemma.
Indeed, we have:
    \[
    \tau(w+r)=a(w+r)+(f+\tau(r)-ar).
    \]
We remark that the element $w+r$ of $\F$ cannot be in $C$, since $w\not\in\K_0$, therefore $f+\tau(r)-ar$ cannot be $0$.
With this in mind, one sees that for any $\a\in C$, 
$\a\neq 0,1$, and any $m\in\Z$ we have:
\begin{equation*}
   \tau\l(\frac{1}{(t-\a)^m}\r)-\frac{a}{(t-\a)^m}=\frac{(1+t)^m}{((1-\a)t-\a)^m}-\frac{a}{(t-\a)^m}.
\end{equation*}
Notice that $\tau\l(\frac{\a}{1-\a}\r)=\a$, hence the poles of the right-hand side of the expression above are on the same $\tau$-orbit. 
Therefore, replacing $w$ by $w+r$, where $r$ is a rational function,  
we can ``shift the poles $\neq 0,1$ of $f$ along their $\tau$-orbit".
The poles of the form $\frac{1}{n}$, for any integer $n\geq 2$, can all be moved 
to the pole $1$, since $\tau(\frac{1}{n})=\frac{1}{n-1}$.  
We conclude that, adding a convenient rational function to $w$,
we can suppose that $f$ has only one pole in each $\tau$-orbit, 
that does not contain $0$. 
We point out that the change of unknown function that we performed does not change the fact that $f$ 
does not have a pole at $\infty$.
\end{enumerate}
{Finally, if $f=0$ and $a\neq 1$, the functional equation $\tau(y)=ay$ cannot have a solution 
$w\in\F$, as noticed in Example~\ref{exa:exponentials}. On the other hand, if $a=1$, the solution $w$ is a constant, 
hence $w\in\K_0$.}
\par 
{From now on we suppose that $f\neq 0$. Under the assumptions \ref{itone}, \ref{ittwo} and \ref{itthree} above,
we now assume by contradiction that $w\in\F\setminus\K_0$ is differentially algebraic over $\K$.}
Let us consider the total ring of fractions $\F_a$ of the $\partial$-Picard-Vessiot ring of $\tau(y)=ay$ over the field 
$\F$ (see Proposition~\ref{prop:existence of PV-Wibmer}). 
Then $\F_a$ contains a copy of the total Picard-Vessiot ring $F_a$ of $\tau(y)=ay$ over $\K$, which is a field, by 
{Lemma~\ref{lemma:constat-a-PV} above.}
Since $w$ is differentially algebraic over $\K$, 
it is differentially algebraic over $F_a$ and  
we can consider the functional equation 
$\tau\left(\frac{w}{z}\right)=\frac{w}{z}+\frac{f}{z}$, where $z\in F_a$ satisfies $\tau(z)=az$. 
It follows from
Proposition~\ref{prop:rank1homogeneous-telescopers} that there exist an integer $n\geq 0$,  
and elements $\alpha_0,\dots,\alpha_n\in C$ with $\alpha_n\neq 0$, and $g\in F_a$, such that 
    \[
    \a_0\frac{f}{z}+\alpha_1\partial\left(\frac{f}{z}\right)+\dots+\alpha_n\partial^n\left(\frac{f}{z}\right)=
    \tau\left(\frac{g}{z}\right)-\frac{g}{z}. 
    \]
Since $\partial(z)=cz$ for some $c\in C$, we have $\partial\left(\frac{f}{z}\right)=\frac{\partial(f)}{z}-c\frac{f}{z}$, 
with $c\neq 0$ if $a\neq 1$ and $c=0$ otherwise. 
Calculating recursively $\partial^i\left(\frac{f}{z}\right)$, 
one proves that there exist $\beta_0,\dots,\beta_n\in C$, with $\beta_n=\alpha_n\neq 0$, such that 
    \[
    \beta_0f+\beta_1\partial(f)+\dots+\beta_n\partial^n(f)=\tau(g)-ag.
    \]
If $a=1$, then $g\in F_a=\K$. 
If $a\neq 1$, then $z$ is transcendental over $\K$ and $g\in F_a=\K(z)\subset\K((z))$, therefore we can write 
$g$ as a Laurent series $g=\sum_ig_iz^i\in\K((z))$. 
Plugging $g$ in the identity above, we find:
    \[
    \beta_0f+\beta_1\partial(f)+\dots+\beta_n\partial^n(f)=\sum_i\left(\tau(g_i)a^i-ag_i\right)z^i.
    \]
Since the left-hand side of such an identity is in $\K_0$, the right-hand side must be in $\K_0$ too, 
hence $g=g_0\in\K$.
Summarizing, both for $a=1$ and $a\neq 1$, we have 
$\beta_0f+\beta_1\partial(f)+\dots+\beta_n\partial^n(f)=\tau(g)-ag$, for some $g\in\K$.
Notice that $\beta_0f+\beta_1\partial\l(f\r) +\dots+\beta_n\partial^n\l(f\r)$ has the same non-zero poles as $f$, 
therefore it has at most one pole per non-zero orbit.
\par 
By assumption, $g$ is analytic in a punctured disk around zero.
Let us suppose by contradiction that there exists a singularity
$t_0\in C^\ast$ on the border of the domain of analyticity of $g$, i.e.\ that $g$ is not analytic on $C^\ast$.
We notice that $\tau^{-m}(t_0)=\frac{t_0}{1-mt_0}$ for any $m\in\Z$, $m\geq 0$, 
therefore the orbit of $t_0$ has an accumulation point at~$0$ as $m\to\infty$.
Since $g$ is analytic in the open punctured disk of center $0$ and radius $|t_0|$, 
possibly replacing $t_0$ by another singularity in its orbit,   
we can suppose that, for any positive integer $m$, no $\tau^{-m}(t_0)$ is a singularity of $g$. 
Then $t_0$ is a singularity of $g$ but not of $\tau(g)$, 
while $\tau^{-1}(t_0)$ is a singularity of $\tau(g)$ but not of~$g$, therefore  
$\tau(g)-g$ is forced to have a singularity both at $t_0$ and at $\tau^{-1}(t_0)$, in contradiction with the fact that $f$ has at most one 
singularity in each $\tau$-orbit.
We conclude that the domain of analyticity of $g$ is the whole~$C^\ast$ and that $f$ cannot have any pole other than $0$ and $\infty$.
In other words, $f\in C[t,t^{-1}]$ and $g$ is analytic over~$C^\ast$, with a pole at $0$.
\par
We now notice that for any integer $m\geq 1$, we have
\[
\tau\l(\frac{1}{t^m}\r)-\frac{a}{t^m}=\frac{1-a}{t^m}+
\sum_{j=0}^{m-1}\binom{m}{j}\frac{1}{t^j}.
\]
This means that replacing $w$ by $w+r$ for a convenient choice of $r\in\K_0$,
we can replace $f$ by $f+\tau(r)-ar$ and assume that $f$ is a non-zero polynomial in $t$ without constant coefficient.
We are finally reduced to an $f\in tC[t]$, which implies that $f=0$, since we have supposed that $f$ does not have any pole at $\infty$.
As we have noticed at the beginning, we can never obtain a zero inhomogeneous term in the functional equation by adding a rational function $r$ to $w$, 
therefore we have found a contradiction. This means that $w$ cannot be differentially algebraic over $\K$, unless it is rational.
\end{proof}

We have finished the core of the proof. It only remains to put the pieces together.  

\begin{lemma}\label{lemma:homogenouus-ADH++}
{{(\cite[Prop.~3.9]{HardouinSinger})}
}
Let $a\in\K_0$, $a\neq 0$, and let $F_a$ be a $\K$-algebra as in Notation
 \ref{notation:algebraF} containing a solution $z$ of $\tau(y)=ay$.
Then either $z$ is differentially algebraic over $\K_0$ and there exist $a^*\in C$ and $b\in\K_0$, 
{both non-zero}
such that $a=a^*\frac{\tau(b)}{b}$,
or the solution $z$  (and hence any solution)  of $\tau(y)=ay$ is differentially transcendental over $\K$.
\end{lemma}

{Example~\ref{exa:exponentials} above shows that, if $z$ is differentially algebraic, it is essentially an exponential function. The following example illustrates the case when $z$ is differentially transcendental.}

\begin{exa}
We set $C=\C$ and we go back to Example~\ref{exa:Bell-homogeneous}. 
Since we know that the gamma function is not rational, we
immediately obtain that $\Ga\l(\frac{1}{t}\r)^{-1}$ 
is strongly differentially transcendental, being a solution of  
$\tau(y)=ty$.
Notice that the statement makes sense since $\Ga\l(\frac{1}{t}\r)^{-1}$ 
is analytic in any 
punctured disk around $0$, hence we can consider the $\K$-field of functions $\K\l(\Ga\l(\frac{1}{t}\r)^{-1}\r)$.
\end{exa}

\begin{proof}[Proof of Lemma~\ref{lemma:homogenouus-ADH++}]
If $z$ is differentially algebraic over $\K$, so is $\partial(z)/z$, which is solution of $\tau(y)=y+\partial(a)/a$.
It follows from Lemma~\ref{lemma:inhomogenouus-ADH++}
that $\partial(z)/z$ is differentially algebraic over $\K$ if and only if
$v\coloneqq  \partial(z)/z\in\K_0$. We have $\partial(a)/a=\tau(v)-v$.
Since both $a$ and $v$ are rational functions, we make the change of variable $s=\frac{1}{t}$, so that 
$\tau(s)=s+1$ and $\partial=-\frac{d}{ds}$.
As we have remarked in Example~\ref{exa:Bernoulli}, if $v$, as a function of the variable $s$, 
has a pole in $C$, then $\tau(v)-v$ must have at least two poles in the same orbit.
In particular, if $v$ has a pole of order greater than $1$, then $\tau(v)-v$ has at least two poles 
of order greater than $1$, while $\frac{\partial(a)}{a}$ has none. Therefore $v$ does not have any finite 
pole of order greater than $1$. Since all the residues of $\frac{\partial(a)}{a}$ are integers, 
the same argument shows that $v$ cannot have any simple pole with a residue in $C\setminus\Z$.
We conclude that $v$ is the sum of a polynomial $p\in C[t]$ and of some terms of the form 
$\frac{m}{s-a}$, with $m\in\Z$ and $a\in C$. Then $\tau(p)-p\in C[t]$, while $\frac{\partial(a)}{a}$ 
is the logarithmic derivative of a rational function. Therefore $\tau(p)-p=0$ and $p\in C$.
Since $v$ is determined up to a constant, we can suppose that $p=0$. 
Hence there exists a rational function $b$ such that $v=\frac{\partial(b)}{b}$. 
We conclude that $\tau\l(\frac{\partial(z)}{z}-\frac{\partial(b)}{b}\r)=\frac{\partial(z)}{z}-\frac{\partial(b)}{b}$, for some 
$b\in\K_0$, hence that there exists $a^*\in C$ such that $a=a^*\frac{\tau(b)}{b}$.
\end{proof}

\paragraph*{Proof of Theorem \ref{thm:ADH++}}
Let $F_a$ be the total ring of fractions of the $\partial$-Picard-Vessiot ring of $\tau(y)=ay$ over $\K$, constructed in
Proposition~\ref{prop:existence of PV-Wibmer} and let $z\in F_a$ be a solution of $\tau(y)=ay$.
If $z$ is differentially transcendental over $\K$, then~$w$ is differentially transcendental over $\K$
because of Theorem \ref{thm:mainTHM}, therefore there is nothing more to prove.
If $z\in F_a$ is differentially algebraic over $\K$ then Lemma~\ref{lemma:homogenouus-ADH++} 
implies that there exist $a^*\in C$ and $b\in\K_0$ suh that $a=a^*\frac{\tau(b)}{b}$. 
Therefore $w/b\in\F$
is differentially algebraic over $\K$ and $\tau(w/b)=a^*(w/b)+f/(ab)$. 
Without loss of generality, we can write $w$ for $w/b$, $a$ for $a^*$ and $f$ for $f/(ab)$, 
so that $\tau(w)=aw+f$, with $a\in C\setminus\{0\}$ and $f\in\K_0$. 
This is the situation of Lemma~\ref{lemma:inhomogenouus-ADH++},
which allows to conclude that $w\in\K_0$.
\qed

\section*{References}

\newcommand{\etalchar}[1]{$^{#1}$}
\def\cprime{$'$}

\renewcommand{\notesname}{Remarks for the referees\\
\bfseries N.B. Correction 17 corresponds to the 17-th remark of the report of the first referee. 
The two remarks of the second referee are commented without numbering them.}


\end{document}